\documentclass[10pt]{amsart}

\usepackage[T1]{fontenc} 
\usepackage[utf8]{inputenc} 
\usepackage{longtable}
\usepackage{amssymb}
\usepackage{amsthm}
\usepackage{amsmath}
\usepackage{wasysym}
\usepackage{dcpic, pictex}
\usepackage{bbm}
\usepackage{enumerate}
\usepackage{amsfonts}
\usepackage{amsthm}
\usepackage{amsmath}
\usepackage{times}
\usepackage{amscd}
\usepackage[mathscr]{eucal}
\usepackage{indentfirst}
\usepackage{pict2e}
\usepackage{epic}
\usepackage{epstopdf} 
\usepackage{verbatim}
\usepackage{cancel}
\usepackage[foot]{amsaddr}
\usepackage{lipsum}
\usepackage{mathrsfs}
\usepackage{bm}
\usepackage{stmaryrd}

\usepackage{
amsfonts, 
latexsym, 
enumerate,
}
\usepackage[english]{babel}
\allowdisplaybreaks
\makeindex
\newtheorem{theorem}{Theorem}[section]
\newtheorem{lemma}[theorem]{Lemma}
\newtheorem{proposition}[theorem]{Proposition}

\newtheorem{example}[theorem]{Example}

\newtheorem{assumption}[theorem]{Assumption}
\usepackage[a4paper,top=3.5cm,bottom=3.5cm,left=3cm,right=3cm]{geometry}
\numberwithin{equation}{section}

\allowdisplaybreaks

\usepackage{ bbold }
\usepackage{xcolor}

\newcommand{\dif}{\mathrm{d}}
\newcommand{\E}{\mathbf{E}}
\newcommand{\R}{\mathbf{R}}
\newcommand{\C}{\mathbf{C}}
\newcommand{\Z}{\mathbf{Z}}

\newcommand{\N}{\mathbf{N}}

\newcommand{\ii}{\mathrm{i}}
\newcommand{\ee}{\mathrm{e}}

\usepackage{hyperref}

\title[Spectral independence of almost fully correlated random matrices]{Spectral independence of almost fully correlated random matrices}

\date{\today}

\begin{document}

\maketitle

\vspace{0.25cm}

\renewcommand{\thefootnote}{\fnsymbol{footnote}}

\noindent
\mbox{}%
\hfill%
\begin{minipage}{0.21\textwidth}
	\centering
	{Oleksii Kolupaiev}\footnotemark[1]\\
	\footnotesize{\textit{Oleksii.Kolupaiev@ist.ac.at}}
\end{minipage}
\hfill%
\mbox{}%
\footnotetext[1]{Institute of Science and Technology Austria, Am Campus 1, 3400 Klosterneuburg, Austria. 
}
\footnotetext[1]{Supported by the ERC Advanced Grant ``RMTBeyond'' No.~101020331.}

\renewcommand*{\thefootnote}{\arabic{footnote}}
\vspace{0.25cm}

\begin{abstract} 
We study the joint spectral properties of two coupled random matrices $H^{(1)}$ and $H^{(2)}$, which are either real symmetric or complex Hermitian. The entries of these matrices exhibit polynomially decaying correlations, both within each matrix and between them. Surprisingly, we find that under extremely weak decorrelation condition, permitting $H^{(1)}$ and $H^{(2)}$ to be almost fully correlated, the fluctuations of their individual eigenvalues in the bulk of the spectrum are still asymptotically independent. Furthermore, we demonstrate that this decorrelation condition is optimal.
\end{abstract}
\vspace{0.15cm}


\vspace{0.25cm}
\normalsize

\section{Introduction}

Let $H$ be an $N\times N$ real symmetric or complex Hermitian matrix with random entries. The asymptotic behavior of the eigenvalues of $H$ as $N\to\infty$ has been extensively studied over the past several decades under increasingly general conditions on $H$. However, much less is known about the joint eigenvalue distribution of two random matrices that are correlated. This question has been primarily studied in the context of (invariant) \emph{multi-matrix models}, where the joint distribution of several matrices is given in a closed form; see for example \cite{Eynard97, Eynard98, Filev13, multi_universality}. In particular, Figalli and Guionnet show in \cite{multi_universality} that for a general multi-matrix model, the joint local eigenvalue statistics exhibits universal behavior when the coupling parameter is sufficiently small, meaning that the eigenvalues locally behave as if they were the eigenvalues of independent GUE matrices. Joint properties of multiple random matrices have also been studied in the context of multi-time correlation functions of the \emph{Dyson Brownian motion} (DBM) \cite{fixed_E} and Unitary Brownian motion \cite{LQG}. Moreover, random matrices that are correlated arise naturally in the setting of \emph{minor processes}, where the random matrices represent the nested principal minors of a common ambient random matrix \cite{determinantal_corr, Gor15, Gor_dbm, minor_process, LFL_Wigner}. 

Apart from \cite{minor_process, LFL_Wigner}, these results rely on the explicit formulas for the joint probability distribution of the eigenvalues of multiple matrices. In this work we study a Wigner-like analogue of a two-matrix model, where instead of assuming a closed form of the joint distribution, we impose general conditions on the correlation structure of the matrix entries. Inspired by \cite{multi_universality}, we move towards understanding the joint spectral properties of a correlated random matrix pair by addressing the following natural question.
\vspace{2mm}

\noindent \textbf{Q.:} How strongly can a pair of random matrices be correlated, while still preserving the asymptotic independence of their local spectral statistics (i.e, of the fluctuations of individual eigenvalues)?
\vspace{2mm}

\noindent We answer this question in Theorem \ref{theo:main} for a broad class of matrix pairs, including correlations between the matrices as well as between the entries of each of them. To illustrate this result, we first consider the following simple setting. Let 
\begin{equation}
W^{(1)}=\big(w_{ab}^{(1)}\big)_{a,b=1}^N\quad\text{and}\quad W^{(2)}=\big(w_{ab}^{(2)}\big)_{a,b=1}^N
\label{eq:intro_W}
\end{equation}
be $N\times N$ random matrices of the same symmetry type. Assume that $W^{(1)}$ and $W^{(2)}$ are \emph{Wigner-type} matrices, meaning that their entries are centered and independent up to the symmetry constraint. Furthermore, suppose that the correlations between $w_{ab}^{(1)}$ and $w_{cd}^{(2)}$ may occur only if $\{a,b\}=\{c,d\}$, and that for any indices $a$ and $b$ we have
\begin{equation}
\big|\E w_{ab}^{(1)}w_{ab}^{(2)}\big| \le (1-\alpha) \left(\E \big|w_{ab}^{(1)}\big|^2\right)^{1/2} \left(\E \big|w_{ab}^{(2)}\big|^2\right)^{1/2},\quad \forall\, 1\le a,b\le N,
\label{eq:intro_decorr}
\end{equation}
for some $\alpha\in [0,1]$. For $j=1,2$ we take an $N\times N$ deterministic matrix $A^{(j)}$ of the same symmetry type as $W^{(j)}$, and set $H^{(j)}:=A^{(j)}+W^{(j)}$. Additionally, fix a point $E^{(j)}$ in the bulk of the spectrum of $H^{(j)}$. We assume that $H^{(1)}$ and $H^{(2)}$ are at least slightly decorrelated, specifically, that \eqref{eq:intro_decorr} holds for some $\alpha\gg N^{-1}$. Under this remarkably weak assumption we prove that the fluctuations of several eigenvalues of $H^{(1)}$ closest to $E^{(1)}$ are asymptotically independent from the fluctuations of several eigenvalues of $H^{(2)}$ closest to $E^{(2)}$. More precisely, we show that the joint eigenvalue correlation functions of these matrices, rescaled according to the microscopic regime, factorize into the product of  the individual correlation functions of $H^{(1)}$ and $H^{(2)}$, up to a small error term. This result is somewhat surprising: the constraint $\alpha\gg N^{-1}$ allows $H^{(1)}$ and $H^{(2)}$ to be nearly identical, while their local eigenvalue statistics are still independent. The threshold $\alpha=N^{-1}$ is sharp: for $\alpha\ll N^{-1}$, there exist examples of matrix pairs satisfying \eqref{eq:intro_decorr} with strongly coupled eigenvalues. We discuss this optimality in more detail later in Example~\ref{ex:optimal}.

Our analysis of the joint eigenvalue statistics extends beyond the basic Wigner-type model presented in the previous paragraph. To capture more intricate correlation structures, we apply a \emph{linear filtering} to the matrices $W^{(1)}$ and $W^{(2)}$ defined in \eqref{eq:intro_W}. This involves applying a pair of deterministic linear transformations to these matrices, resulting in $H^{(1)}$ and $H^{(2)}$. This procedure introduces non-trivial correlations between the entries both within and between the matrices. Assuming that \eqref{eq:intro_decorr} holds for $W^{(j)}$, $j=1,2$, we prove in Theorem \ref{theo:main} that the same conclusion as above applies to this more general setting: as long as $\alpha\gg N^{-1}$, the fluctuations of individual eigenvalues in the bulk of $H^{(1)}$ and $H^{(2)}$ are almost independent.

The smallness of the threshold $\alpha=N^{-1}$, which separates the asymptotic independence ($\alpha\gg N^{-1}$) and the strong coupling ($\alpha\ll N^{-1}$) of eigenvalue fluctuations arises from the \emph{locality} of eigenvalue statistics we consider. By focusing only on finitely many eigenvalues as $N$ goes to infinity, we restrict ourselves to the \emph{microscopic regime}. This contrasts with the \emph{mesoscopic} and \emph{macroscopic regimes}, which concern the collective behavior of $N^\varepsilon$ and $\varepsilon N$ eigenvalues, respectively, for some fixed $\varepsilon>0$. In the macroscopic regime, the answer to the question (\textbf{Q}) is different and in some sense more natural: the global eigenvalue statistics of $H^{(1)}$ and $H^{(2)}$ are asymptotically independent only when $|1-\alpha|\ll 1$, indicating that these matrices are nearly independent. For instance, construct $(H^{(1)},H^{(2)})$ saturating \eqref{eq:intro_decorr} for some $\alpha\in [0,1]$ by
\begin{equation}
H^{(j)}=\sqrt{1-\alpha}W_0 + \sqrt{\alpha}W_j,\quad j=1,2,
\label{eq:intro_ex}
\end{equation}
where $W_0$, $W_1$ and $W_2$ are independent GOE matrices. Let $\{ \lambda_n^{(j)}\}_{n=1}^N$ denote the eigenvalues of $H^{(j)}$ listed in increasing order. As a very simple example, consider the global eigenvalue statistics
\begin{equation*}
\xi_j:=\mathrm{Tr}\left[\big(H^{(j)}\big)^2\right] = \sum_{n=1}^N \big(\lambda_n^{(j)})^2,\quad j=1,2.
\end{equation*}
A straightforward calculation shows that the correlation between the random variables $\xi_1$ and $\xi_2$ equals to
\begin{equation*}
\mathrm{corr}(\xi_1,\xi_2):=\frac{\E [\xi_1\xi_2]-\E[\xi_1]\E[\xi_2]}{\sqrt{\E [\xi_1^2]-\E[\xi_1]^2}\sqrt{\E [\xi_2^2]-\E[\xi_2]^2}}=(1-\alpha)^2.
\end{equation*}
In particular, $\xi_1$ and $\xi_2$ are decorrelated only when $|1-\alpha|\ll 1$.

Now we compare our result with the one obtained in \cite{multi_universality}, where the correlation strength between the matrices is governed by the coupling parameter $a$. The simplest case which fits the invariant set-up of \cite{multi_universality}, is the special Gaussian example where the joint density function of the complex Hermitian matrices $H^{(1)}$ and $H^{(2)}$ is proportional to
\begin{equation}
\mathrm{exp}\left\lbrace -\frac{N}{2}\mathrm{Tr} \left(a \big(H^{(1)}-H^{(2)}\big)^2 + \big(H^{(1)}\big)^2+\big(H^{(2)}\big)^2\right)\right\rbrace.
\label{eq:intro_integr}
\end{equation}
It has been shown in \cite{multi_universality} that for some $N$-independent constant $a_0>0$, the joint local eigenvalue correlation functions of $H^{(1)}$ and $H^{(2)}$ asymptotically match those of the independent GUE matrices once $\vert a\vert\le a_0$. Remarkably, joint correlations of up to $N^{2/3}$ bulk eigenvalues can be simultaneously monitored. Additionally assuming that $a$ is positive, the matrices $H^{(1)}$ and $H^{(2)}$ saturate the decorrelation inequality \eqref{eq:intro_decorr} with the control parameter $\alpha\sim (1+a)^{-1}$. In contrast to \cite{multi_universality}, our work explores a much wider range of positive coupling parameters whose analogue in \cite{multi_universality} would be $0\le a\le N^{1-\epsilon}$, thus extending well beyond the perturbative regime $0\le a\le a_0$. However, this generality comes at the cost of restricting to a more classical setting, when only joint correlations of $k$ eigenvalues are monitored for any fixed $k$ independent of $N$.  

The proof of our main result, Theorem \ref{theo:main}, follows the three step strategy outlined in \cite{erdHos2017dynamical}. In the first step, we analyze each matrix separately, without addressing their joint properties. We borrow a so-called \emph{local law} from \cite{slow_corr, cusp_univ}, which provides a concentration bound for the resolvent $(H^{(j)}-z)^{-1}$, where $z\in \C\setminus\R$, around its deterministic counterpart, later denoted by $M^{(j)}(z)$. In the second step we prove Theorem \ref{theo:main} for matrices with sufficiently large independent Gaussian components, relying on the fixed energy universality of the DBM from \cite{fixed_E}. Our main contribution lies in the third step, where we remove these components using the continuity argument known as the Green Function Theorem (GFT). Since we impose an extremely weak decorrelation assumption on $(H^{(1)}, H^{(2)})$, the joint correlation structure of these matrices is highly degenerate. As a result, the GFT needs to be reconstructed, since existing GFT continuity arguments highly rely on a certain non-degeneracy condition (called \emph{fullness} in \cite{slow_corr}) of the correlation structure. This condition ensures that there exists a random matrix (or a pair of them, depending on the setting) with the same second-order correlation structure, having an independent GUE/GOE component of order 1. In the absence of this condition, we develop a new matrix-valued comparison flow that preserves the second order joint correlation structure of $H^{(1)}$ and $H^{(2)}$ while efficiently adding sufficiently large independent GUE/GOE components to them. For more details see Section \ref{sec:plan}. Our approach applies not only to a pair of coupled random matrices, but also to a single random matrix. We believe it may be useful in other problems involving random matrices with highly correlated entries.

\subsection*{Notations and conventions}
We set $[k] := \{1, ... , k\}$ for $k \in \N$ and $\langle A \rangle := N^{-1} \mathrm{Tr}(A)$, $N \in \N$, for the normalized trace of an $N \times N$-matrix $A$. For positive quantities $f, g$ we write $f \lesssim g$, $f \gtrsim g$, to denote that $f \le C g$ and $f \ge c g$, respectively, for some $N$-independent constants $c, C > 0$ that depend only on the basic control parameters of the model in Assumptions~\ref{ass:W-type} and \ref{ass:model} below. In informal explanations, we frequently use the notation $f\ll g$, which indicates that $f$ is "much smaller" than $g$. 

We denote vectors by bold-faced lower case Roman letters $\boldsymbol{x}, \boldsymbol{y} \in \C^{N}$, for some $N \in \N$. Moreover, for vectors $\boldsymbol{x}, \boldsymbol{y} \in \C^{N}$ we define
 \begin{equation*}
	\langle \boldsymbol{x}, \boldsymbol{y} \rangle := \sum_i \bar{x}_i y_i. 
\end{equation*}
Matrix entries are indexed by lower case Roman letters $a, b, c , ... ,i,j,k,... $ from the beginning or the middle of the alphabet and unrestricted sums over those are always understood to be over $\{ 1 , ... , N\}$. 

Finally, we will use the concept  \emph{with very high probability},  meaning that for any fixed $D > 0$, the probability of an $N$-dependent event is bigger than $1 - N^{-D}$ for all $N \ge N_0(D)$. We will use the convention that $\xi > 0$ denotes an arbitrarily small positive exponent, independent of $N$.
 Moreover, we introduce the common notion of \emph{stochastic domination} (see, e.g., \cite{loc_sc_gen}): For two families
\begin{equation*}
	X = \left(X^{(N)}(u) \mid N \in \N, u \in U^{(N)}\right) \quad \text{and} \quad Y = \left(Y^{(N)}(u) \mid N \in \N, u \in U^{(N)}\right)
\end{equation*}
of non-negative random variables indexed by $N$, and possibly a parameter $u$, we say that $X$ is stochastically dominated by $Y$, if for all $\epsilon, D >0$ we have 
\begin{equation*}
	\sup_{u \in U^{(N)}} \mathbf{P} \left[X^{(N)}(u) > N^\epsilon Y^{(N)}(u)\right] \le N^{-D}
\end{equation*}
for large enough $N \ge N_0(\epsilon, D)$. In this case we write $X \prec Y$. If for some complex family of random variables we have $\vert X \vert \prec Y$, we also write $X = O_\prec(Y)$.

\textbf{Acknowledgment:} The author is grateful to L{\' a}szl{\' o} Erd{\H o}s and Giorgio Cipolloni for suggesting the project and fruitful discussions.

\section{Set-up and main results}

\subsection{Model} In this section, we introduce the model of a pair of Hermitian random matrices that are correlated, which is the central object of study in this paper. The discussion is divided into two parts. In Section~\ref{subsec:W-type} we begin with the simplest setting, where each matrix has independent entries (up to the Hermitian symmetry) and the correlations between the two of them are relatively straightforward. Later in Section~\ref{subsec:correlated} we extend this framework to include matrices with more intricate correlations both within and between them. Since the model constructed in Section \ref{subsec:correlated} is the generalization of the one in Section~\ref{subsec:W-type}, our main analysis is performed within this broader framework. However, readers not interested in the general correlated case may skip Section \ref{subsec:correlated} to avoid technical details. At the end of Section \ref{subsec:W-type} we explain how to follow the rest of the paper using only the simpler framework introduced there.

\subsubsection{A pair of correlated Wigner-type matrices}\label{subsec:W-type} Consider a pair of random matrices $(W^{(1)}, W^{(2)})$, which are either both real symmetric or both complex Hermitian. Their entries are centered and normalized so that their typical size is of order $N^{-1/2}$. We assume that within each matrix the entries are independent up to the symmetry constraint, and that the only non-zero correlations between $W_{ab}^{(1)}$, $W_{cd}^{(1)}$ with $a\le b$, $c\le d$, occur in the case $a=c$, $b=d$. Additionally, consider a pair of deterministic matrices $(A^{(1)}, A^{(2)})$ of the same symmetry type as $(W^{(1)}, W^{(2)})$, and define 
\begin{equation}
H^{(j)}:=A^{(j)}+W^{(j)},\quad j=1,2.
\label{eq:model_basic}
\end{equation}
This forms the basic model for the pair of correlated random matrices, which is covered by the analysis performed in the following sections. We now state the precise assumptions on the components of this model. 

\begin{assumption}\label{ass:W-type}
There exist constants $c_0,C_0>0$ such that the following holds for all sufficiently large positive integers $N$.

\noindent \textbf{\emph{(i)}} \emph{(Boundedness of the deformation matrices)} $\lVert A^{(j)}\rVert\le C_0$ for $j=1,2$.

\noindent \textbf{\emph{(ii)}} \emph{(Mean-field regime)} We assume that $W^{(j)}=\big(w^{(j)}_{ab}\big)_{a,b=1}^N$ is a Wigner-type matrix (a concept initially introduced in \cite[Section 1.1]{univ_W-type}), meaning that the entries of $W^{(j)}$ are independent up to the symmetry constraint and for $a,b\in[N]$ we have
\begin{equation}
\E w^{(j)}_{ab}=0,\quad \frac{c_0}{N}\le \E \big| w^{(j)}_{ab}\big|^2\le\frac{C_0}{N}.
\label{eq:init_flat}
\end{equation} 
If $W^{(j)}$ is complex, we further assume that $\Re w^{(j)}_{ab}$ and $\Im w^{(j)}_{ab}$ are independent and that $\E\big(w^{(j)}_{ab}\big)^2 = 0$ for $a\neq b$. Additionally, for any $p\in\N$ there exists a constant $C_p>0$ such that
\begin{equation}
\E\big| w^{(j)}_{ab}\big|^p \le C_p N^{-p/2}.
\label{eq:p_moment}
\end{equation}

\noindent \textbf{\emph{(iii)}} \emph{(Correlations between the matrices)} The entries $w^{(1)}_{ab}$ and $w^{(2)}_{cd}$ are independent for $\{a,b\}\neq \{c,d\}$. Moreover, there exist a (small) constant $\epsilon>0$ and an $N$-dependent control parameter $\alpha\ge N^{-1+\epsilon}$ such that
\begin{equation}
\left\vert\E w^{(1)}_{ab} w^{(2)}_{cd}\right\vert \le (1-\alpha)\left(\E \big| w^{(1)}_{ab}\big|^2\right)^{1/2} \left(\E \big| w^{(2)}_{cd}\big|^2\right)^{1/2},\quad \{a,b\}=\{c,d\}.
\label{eq:decorr}
\end{equation}
In the complex Hermitian case we additionally assume that $\Im w^{(1)}_{ab}$ is independent of $\Re w^{(2)}_{cd}$ for all index choices, and similarly $\Re w^{(1)}_{ab}$ is independent of $\Im w^{(2)}_{cd}$. We also require that \eqref{eq:decorr} holds separately for the real and imaginary parts of the matrix entries. 
\end{assumption}

The main parameter of the model \eqref{eq:model_basic} is $\alpha\in[0,1]$ introduced in Assumption \ref{ass:W-type}(iii). It controls the correlation strength between $W^{(1)}$ and $W^{(2)}$: the smaller is $\alpha$, the stronger may be the correlation. Later in Example \ref{ex:optimal} we explain the significance of the threshold $\alpha\gg N^{-1}$.

In the next section we generalize the model \eqref{eq:model_basic} by applying a linear operator $\Phi^{(j)}$ to $W^{(j)}$, where $\Phi^{(j)}$ maps matrices to matrices. For simplicity, the reader may skip Section \ref{subsec:correlated} and assume that $\Phi^{(j)}$ acts as the identity, which recovers the setting described in this section.

\subsubsection{General correlated case}\label{subsec:correlated} Before generalizing the model introduced in \eqref{eq:model_basic}, we discuss the correlation structure of a single random matrix. Let $W$ be an $N\times N$ random matrix, either real symmetric or complex Hermitian, whose entries are centered and may exhibit non-trivial correlations beside the symmetry condition. The covariance tensor $\Sigma_W$ associated to $W$ is defined through its action on $N\times N$ deterministic matrices
\begin{equation}
\Sigma_W[R]=\E\left[W\mathrm{Tr}\left[ RW\right]\right],\quad \forall R\in\C^{N\times N}.
\label{eq:cov_tensor}
\end{equation}
This is a non-negative linear operator on $\C^{N\times N}$ equipped with the scalar product $\langle R_1,R_2\rangle:=\mathrm{Tr}\left[R_1^*R_2\right]$. Additionally, $\Sigma_W$ preserves the symmetry, i.e. it maps $\mathrm{Sym}_\beta(N)$ into itself, where $\mathrm{Sym}_\beta(N)$ is the set of $N\times N$ matrices with the same symmetry type as $W$. Here we set $\beta=1$ in the real symmetric case and $\beta=2$ in the complex Hermitian one. Since the action of $\Sigma_W$ on $\mathrm{Sym}_\beta(N)$ fully determines the joint second moments of the entries of $W$, we further restrict the covariance tensor to this subspace. 

Let $\Phi:\mathrm{Sym}_\beta(N)\to\mathrm{Sym}_\beta(N)$ be a non-negative\footnote{Non-negativity means that $\langle R^*\Phi[R]\rangle\ge 0$ for any $R\in\mathrm{Sym}_\beta(N)$.} linear operator. Consider the random matrix $\widetilde{W}:=\Phi[W]$. This transformation, known as \emph{linear filtering}, modifies the covariance tensor of the original random matrix $W$. Specifically, the covariance tensor of $\widetilde{W}$ is given by
\begin{equation}
\Sigma_{\widetilde{W}}=\Phi\Sigma_W\Phi.
\label{eq:cov_t_transform}
\end{equation}
In particular, if $W$ is a GOE/GUE matrix, then $\Sigma_W[R]=2(\beta N)^{-1}R$ for any $R\in\mathrm{Sym}_\beta(N)$, leading to $\Sigma_{\widetilde{W}}=2(\beta N)^{-1}\Phi^2$. Thus, linear filtering provides a method to construct a random matrix with any covariance tensor that satisfies the conditions listed below \eqref{eq:cov_tensor}.

Now we turn to the discussion of a pair of correlated random matrices. Let $(W^{(1)}, W^{(2)})$ be a pair of correlated Wigner-type matrices satisfying Assumption \ref{ass:W-type}, which we refer to as the \emph{initial signal}. Let $\Phi^{(1)},\Phi^{(2)}: \mathrm{Sym}_\beta(N)\to\mathrm{Sym}_\beta(N)$ be any non-negative \emph{filtering transformations}.  Applying $\Phi^{(1)}$, $\Phi^{(2)}$ to the initial signal, we obtain the \emph{filtered signal}, which is the correlated matrix pair $\widetilde{W}^{(j)}:=\Phi^{(j)}[W^{(j)}]$, $j=1,2$. This linear filtering introduces non-trivial correlations both within each matrix and between the two matrices. 

Next, consider deformations $A^{(j)}\in\mathrm{Sym}_\beta(N)$ satisfying Assumption \ref{ass:W-type}(i), and set
\begin{equation}
H^{(j)}:=A^{(j)}+\widetilde{W}^{(j)}=A^{(j)}+ \Phi^{(j)}\big[W^{(j)}\big],\quad j=1,2.
\label{eq:model}
\end{equation}
This model extends the one in \eqref{eq:model_basic} and is the main object of our analysis. In addition to Assumption~\ref{ass:W-type} on $(W^{(1)},W^{(2)})$, we impose the following assumptions on the filtering transformations. 

\begin{assumption}\label{ass:model}  \emph{(Non-degeneracy and locality of filtering transformations)} We assume that $ \Phi^{(j)}\ge c_0$ for some fixed $c_0>0$, i.e. that $\langle R\Phi^{(j)}[R]\rangle\ge c_0\langle R^2\rangle$ holds for any $R\in\mathrm{Sym}_\beta(N)$. Next, for $1\le i\le k\le N$ define an $N\times N$ matrix $E_{ik}$ by 
\begin{equation}
\left(E_{ik}\right)_{ab}:=\delta_{i,a}\delta_{k,b},\quad \left(E_{ik}\right)_{ba}:=\left(E_{ik}\right)_{ab},\quad \forall\, 1\le a\le b\le N.
\label{eq:E}
\end{equation}
We further assume that for some $s>2$ and $C_0>0$ it holds that
\begin{equation}
\left\vert\left(\Phi^{(j)}\left[E_{ik}\right]\right)_{ab}\right\vert\le \frac{C_0}{1+(\vert i-a\vert+\vert k-b\vert)^s}
\label{eq:Phi_decay}
\end{equation}
for all $i\le k$ and $a\le b$. Moreover, in the complex Hermitian case we define $E_{ik}^{\Im}$ for $i<k$ as the matrix $E_{ik}$ with $\left(E_{ik}\right)_{ik}$ multiplied by $\ii$ and $\left(E_{ik}\right)_{ki}$ multiplied by $-\ii$. Finally, we impose an additional condition that \eqref{eq:Phi_decay} remain valid when $E_{ik}$ is replaced with $E_{ik}^{\Im}$.
\end{assumption}

The following example shows that when $\Phi^{(j)}$ is a translation-invariant filtering, the condition $\Phi^{(j)}\ge c_0$ can be replaced by a more easily verifiable one.
\begin{example} Let $\mathbb{T}:=\Z/[N\Z]$, and consider a function $F:\mathbb{T}^2\to\C$ satisfying $F(x,y)=\overline{F}(y,x)$ for all $x,y\in\mathbb{T}$. Define the operator $\Phi:\mathrm{Sym}_{\beta}(N)\to \mathrm{Sym}_{\beta}(N)$ as the convolution
\begin{equation}
(\Phi[R])_{ab}=\sum_{c,d\in\mathbb{T}} F(a-c,b-d)R_{cd},\quad\forall a,b\in\mathbb{T},
\label{eq:Phi_conv}
\end{equation}
where, in the case $\beta=1$, we assume that $F$ is real-valued. Assume additionally that the following bound on the discrete Fourier of $F$ holds for all $\varphi,\theta\in [0,1]$:
\begin{equation}
\widehat{F}(\varphi,\theta):=\sum_{x,y\in\mathbb{T}} F(x,y)\ee^{2\pi i(x\varphi-y\theta)}\ge c_0.
\label{eq:FT}
\end{equation}
Then we have $\Phi\ge c_0$.
\end{example}

Condition \eqref{eq:FT} is reminiscent of \cite[(2.7)]{Gauss_corr}, where a similar bound on the discrete Fourier transform was used to ensure the non-degeneracy of the correlation structure. Note that with $\Phi$ defined in \eqref{eq:Phi_conv}, the filtered matrix $\widetilde{W}:=\Phi[W]$ has a translation-invariant correlation structure on the torus $\mathbb{T}^2$. In particular, $\widetilde{w}_{11}$ may be strongly correlated with $\widetilde{w}_{NN}$, which is not compatible with our assumption \eqref{eq:Phi_decay}. However, a careful examination of our argument shows that the distance $\vert i-a\vert +\vert k-b\vert$ in \eqref{eq:Phi_decay} can be replaced by the distance on $\mathbb{T}^2$ between $(i,j)$ and $(a,b)$, with all results remaining valid after minor adjustments to the proofs.

\subsection{Eigenvalue correlation functions} With the basic notations established, we consider a pair of correlated random matrices $(H^{(1)}, H^{(2)})$ and denote their eigenvalues, ordered increasingly, by $\lambda_1^{(j)}\le \lambda_2^{(j)}\le\ldots\le \lambda_N^{(j)}$, $j=1,2$. For $m,n\in\N$ we define the joint eigenvalue correlation function $p_{m,n}^{(1,2)}$ of $H^{(1)}, H^{(2)}$ implicitly via
\begin{equation}
\E \bigg[ \frac{(N-m)!}{N!}\frac{(N-n)!}{N!}\!\!\!\sum\limits_{\substack{i_1,\ldots, i_m\in [N]\\ j_1,\ldots, j_n\in [N]}}f\left(\lambda_{i_1}^{(1)}, \ldots, \lambda_{i_m}^{(1)}, \lambda_{j_1}^{(2)}, \ldots, \lambda_{j_n}^{(2)} \right) \bigg] = \int_{\R^{m+n}} f({\bf x},{\bf y}) p_{m,n}^{(1,2)} ({\bf x},{\bf y})\dif{\bf x}\dif{\bf y}
\label{eq:joint_cor_func}
\end{equation}
for any smooth test function $f:\R^{m+n}\to\R$ with compact support. The summation in the lhs. of \eqref{eq:joint_cor_func} runs over all sets of indices $(i_1,\ldots,i_m)\in [N]^{m}$ and $(j_1,\ldots,j_n)\in [N]^{n}$ with distinct elements, though some $i$'s may coincide with some $j$'s. Note that while $p_{m,n}^{(1,2)}$ depends on $N$, this dependence is omitted in the notation for simplicity. In the special case $n=0$, \eqref{eq:joint_cor_func} reduces to the standard definition of the eigenvalue correlation function $p^{(1)}_{m}$ of $H^{(1)}$. The normalization in \eqref{eq:joint_cor_func} is chosen in such a way that the integral of $p_{m,n}^{(1,2)}$ over $\R^{m+n}$ equals to 1.

In the extreme case when $H^{(1)}$ and $H^{(2)}$ are independent, their joint eigenvalue correlation function factorizes, i.e.
\begin{equation*}
p_{m,n}^{(1,2)}({\bf x},{\bf y})=p_{m}^{(1)}({\bf{x}})p_{n}^{(2)}({\bf y}).
\end{equation*}
Conversely, in the opposite extreme where $H^{(1)}$ and $H^{(2)}$ are identical, \eqref{eq:joint_cor_func} no longer defines a function but rather a distribution, which we denote by ${\bm p}^{(1,1)}_{m,n}$. For example, for any smooth test function $f:\R^2\to\R$ with compact support we have
\begin{equation*}
{\bm p}^{(1,1)}_{1,1}(f) = \int_{\R^2} f(x,y)p_2^{(1)}(x,y)\dif x\dif y + \frac{1}{N}\int_\R f(x,x) p_1^{(1)}(x)\dif x.
\end{equation*}

In order to analyze the joint eigenvalue correlation functions of $H^{(1)}$ and $H^{(2)}$, it is essential to understand the single-point eigenvalue distributions $p^{(1)}_1$ and $p^{(2)}_1$. To this end, we firstly define the \emph{self-energy operator} $\mathcal{S}^{(j)}$ through its action on $\C^{N\times N}$:
\begin{equation*}
\mathcal{S}^{(j)}[R]:=\E \left[\widetilde{W}^{(j)}R\widetilde{W}^{(j)}\right],\quad \forall R\in\C^{N\times N}, 
\end{equation*}
where we use the notations from \eqref{eq:model}. While $\mathcal{S}^{(j)}$ differs from the covariance tensor of $\widetilde{W}^{(j)}$, it contains the same information about the joint second moments of the entries of $\widetilde{W}^{(j)}$. For any $z\in\C\setminus\R$ the $N\times N$ matrix $M^{(j)}(z)$ is obtained as the unique solution to the \emph{Matrix Dyson Equation} \cite{MDEreview}
\begin{equation}
-M^{(j)}(z)^{-1}=z-A^{(j)}+\mathcal{S}^{(j)}[M^{(j)}(z)],\quad \Im z\Im M^{(j)}(z)>0.
\label{eq:MDE}
\end{equation} 
Next, the \emph{self-consistent density of states} (abbreviated as scDoS) $\rho^{(j)}$ is given by
\begin{equation*}
\rho^{(j)}(z):=\pi^{-1}\vert\langle \Im M^{(j)}(z)\rangle\vert,\, z\in\C\setminus\R,\quad \rho^{(j)}(x):=\lim_{\eta\to+0} \rho^{(j)}(x+\ii \eta),\, x\in\R,
\end{equation*}
and for any $\kappa>0$ we define the $\kappa$-bulk as follows:
\begin{equation*}
\mathbf{B}_\kappa^{(j)}:=\{x\in\R\,:\, \rho^{(j)}(x)\ge\kappa\}.
\end{equation*}
The scDoS $\rho^{(j)}$ approximates the 1-point correlation function $p_1^{(j)}$ in the weak sense. We will justify this approximation later along the proof of Theorem \ref{theo:main}, relying on the results from \cite{slow_corr}.

\subsection{Main result} We are now ready to present our main result, which describes the correlations between fluctuations of individual eigenvalues of $H^{(1)}$ and $H^{(2)}$. The typical spacing between eigenvalues in the bulk of the spectrum is $N^{-1}$. In order to access the joint local eigenvalue statistics of $H^{(1)}, H^{(2)}$ we fix $E^{(j)}\in\mathbf{B}^{(j)}_\kappa$ for $j=1,2$ and some $\kappa>0$. We then rescale the arguments of $p^{(1,2)}_{m,n}$ around $E^{(1)}, E^{(2)}$ by a factor $N$. Our result shows that this rescaled function factorizes up to a small error term, provided $H^{(1)}$ and $H^{(2)}$ are at least slightly decorrelated in the sense of Assumption \ref{ass:W-type}(iii). To keep the main idea simple, we present the assumptions in two alternative forms: the first applies to the simpler model \eqref{eq:model_basic} and requires only Assumption \ref{ass:W-type}, while the second addresses the more general model \eqref{eq:model} and additionally requires Assumption \ref{ass:model}.

\begin{theorem}[Independence of local statistics]\label{theo:main} Assume one of the following two sets of conditions:
\begin{enumerate}
\item\emph{(Wigner-type case)} Matrices $H^{(1)}$, $H^{(2)}$ are given by \eqref{eq:model_basic} and satisfy Assumption \ref{ass:W-type}.
\item\emph{(Correlated case)} Matrices $H^{(1)}$, $H^{(2)}$ are given by \eqref{eq:model} and satisfy Assumptions \ref{ass:W-type} and \ref{ass:model}.
\end{enumerate}
Fix a (small) $\kappa>0$ and take $E^{(j)}\in\mathbf{B}^{(j)}_\kappa$ for $j=1,2$. Then for any $m, n\in\N$ and a compactly supported smooth test function $F\in C_c^{\infty}(\R^{m+n})$ it holds that
\begin{equation}
\begin{split}
&\int_{\R^{m+n}} F({\bf x},{\bf y})\Bigg[ p_{m,n}^{(1,2)}\left(E^{(1)}+\frac{{\bf x}}{\rho^{(1)}(E^{(1)})N}, E^{(2)}+\frac{{\bf y}}{\rho^{(2)}(E^{(2)})N}\right)\\
&\quad - p_{m}^{(1)}\left(E^{(1)}+\frac{{\bf x}}{\rho^{(1)}(E^{(1)})N}\right)p_{n}^{(2)}\left(E^{(2)}+\frac{{\bf y}}{\rho^{(2)}(E^{(2)})N}\right)\Bigg]\dif {\bf x}\dif {\bf y} = \mathcal{O}\left(N^{-c(m,n)}\right), 
\label{eq:goal}
\end{split}
\end{equation}
where $c(m,n)>0$ depends only on $m$ and $n$, and the implicit constant in $\mathcal{O}$ depends on $F$.
\end{theorem}

The following example demonstrates that the threshold $\alpha=N^{-1}$ introduced in Assumption \ref{ass:W-type}(iii) is optimal. Specifically, when $\alpha\gg N^{-1}$, Theorem \ref{theo:main} establishes that the fluctuations of individual eigenvalues of $H^{(1)}$ and $H^{(2)}$ are effectively independent. In contrast, Example \ref{ex:optimal} shows that in the regime $\alpha\ll N^{-1}$ the fluctuations of these eigenvalues become highly correlated for certain choices of $H^{(1)}, H^{(2)}$ which we now describe. 

\begin{example}[Optimality of $N^{-1}$-threshold]\label{ex:optimal} Fix (small) $\epsilon, \kappa>0$. Let $W$, $W^{(1)}_G$, $W^{(2)}_G$ be independent random matrices such that $W$ is a Wigner-type matrix satisfying Assumption \ref{ass:W-type}(ii) and $W^{(1)}_G$, $W^{(2)}_G$ are GUE/GOE matrices of the same symmetry type. For $0<\alpha<N^{-1-\epsilon}$ set 
\begin{equation}
W^{(j)}:=W+\sqrt{\alpha}W^{(j)}_G,\quad j=1,2.
\label{eq:ex_W}
\end{equation}
Take any deterministic $A:=A^{(1)}=A^{(2)}$ satisfying Assumption \ref{ass:W-type}(i) and $\Phi^{(1)}=\Phi^{(2)}$ which act identically on $\mathrm{Sym}_\beta(N)$. In this case $\rho^{(1)}=\rho^{(2)}:=\rho$ and we pick an energy $E\in\mathbf{B}_\kappa(\rho)$. Then for $H^{(1)}$, $H^{(2)}$ constructed in \eqref{eq:model} and for any compactly supported smooth test function $F\in C_c^{\infty}(\R^{m+n})$ it holds that
\begin{equation}
\begin{split}
&\int_{\R^2} F(x,y) p_{1,1}^{(1,2)}\left(E+\frac{x}{\rho(E)N}, E+\frac{y}{\rho(E)N}\right)\dif  x\dif y = \rho(E)\int_\R F(x,x) p_1^{(1)}\left(E+\frac{x}{\rho(E)N}\right)\dif x\\
& \quad+\int_{\R^2} F(x,y)p_2^{(1)}\left(E+\frac{x}{\rho(E)N},E+\frac{y}{\rho(E)N}\right) \dif x\dif y + \mathcal{O}\left(N^{-c}\right). 
\end{split}
\label{eq:optimal}
\end{equation}
In \eqref{eq:optimal} the exponent $c>0$ is universal and the implicit constant in $\mathcal{O}$ depends on $F$.
\end{example}

It is easy to see that $H^{(1)}, H^{(2)}$ constructed in Example \ref{ex:optimal} satisfy all conditions listed in Assumption \ref{ass:W-type} apart from \eqref{eq:decorr}. The proof of \eqref{eq:optimal} follows by a simple perturbative argument and is presented in Appendix~\ref{app:optimal}.

\section{Proof of Theorem \ref{theo:main}}\label{sec:plan}

In this section we prove Theorem \ref{theo:main}, deferring some technical details to later sections. In particular, we show how to handle the highly degenerate correlation structure of $(H^{(1)}, H^{(2)})$. The proof is carried out in three steps with the main novelty lying in the third step. First, we establish a bulk local law for $H^{(j)}$ constructed in \eqref{eq:model}. This step provides bounds on the resolvents of $H^{(1)}$ and $H^{(2)}$ individually but does not capture their joint behavior. This local law serves as an essential input for the next two steps. Then, we prove Theorem \ref{theo:main} in the case when $W^{(1)}$ , $W^{(2)}$ have sufficiently large independent Gaussian components. This is done in Proposition \ref{prop:DBM}. Finally, redesigning a continuity argument known as the Green Function Theorem (GFT), we remove these Gaussian components in Proposition \ref{prop:p_comparison}. 

\begin{proposition}[Local law in the bulk of the spectrum]\label{prop:loc_law} Fix $\kappa, \varepsilon>0$. For $z_j\in\C\setminus\R$, $j=1,2$, denote the resolvent of $H^{(j)}$ by $G^{(j)}(z_j):=(H^{(j)}-z_j)^{-1}$. Then under Assumptions \ref{ass:W-type} and \ref{ass:model} it holds that
\begin{equation}
\left\langle \bm{x}, \left(G^{(j)}(z_j)-M^{(j)}(z_j)\right)\bm{y}\right\rangle \prec \frac{\lVert \bm{x}\rVert\,\lVert \bm{y}\rVert}{\sqrt{N\vert \Im z_j\vert}}\quad \text{and}\quad \left\langle \left(G^{(j)}(z_j)-M^{(j)}(z_j)\right)B\right\rangle \prec \frac{\lVert B\rVert}{N\vert \Im z_j\vert}
\label{eq:loc_law}
\end{equation}
uniformly in $z_1,z_2\in\C$ with $\Re z_j\in\mathbf{B}_\kappa^{(j)}$, $N^{-1+\varepsilon}<\vert\Im z_j\vert<N^{100}$ and deterministic $\bm{x},\bm{y}\in\C^N$, $B\in\C^{N\times N}$.
\end{proposition}
We derive \eqref{eq:loc_law} directly from \cite[Theorem 2.8]{cusp_univ} and postpone the details to Appendix \ref{app:loc_law}. While Proposition \ref{prop:loc_law} is not explicitly referenced further in this section, it plays a crucial role in justifying the second and third steps of our strategy.

The second step of the proof is encapsulated in the following proposition.
\begin{proposition}\label{prop:DBM} Let $H^{(1)}, H^{(2)}$ be defined as in \eqref{eq:model} with ingredients of this model satisfying Assumptions \ref{ass:W-type} and \ref{ass:model}. Fix a (small) $\sigma>0$ and assume that for some $N^{-1+\sigma}\le s\le N^{-\sigma}$ we have
\begin{equation}
W^{(j)}=\widehat{W}^{(j)}+\sqrt{s}\,W_G^{(j)},\quad j=1,2,
\label{eq:g_component}
\end{equation}
where $W_G^{(1)}$, $W_G^{(2)}$ are independent GUE/GOE matrices, which are also independent from the Wigner-type matrices $\widehat{W}^{(1)}$ and $\widehat{W}^{(2)}$ satisfying Assumption \ref{ass:W-type}(ii).  Then \eqref{eq:goal} holds for $H^{(1)}$ and $H^{(2)}$.
\end{proposition}
\noindent The proof of Proposition \ref{prop:DBM} is presented in Section \ref{sec:GD}. 

We now introduce our novel GFT technique and discuss its key features. Combine $W^{(1)}$ and $W^{(2)}$ into a single block-diagonal matrix
\begin{equation*}
W^{(1,2)}=\mathrm{diag}\left(W^{(1)},W^{(2)}\right) :=
\begin{pmatrix}
W^{(1)}&0\\
0&W^{(2)}
\end{pmatrix}.
\end{equation*}
Let $\Sigma^{(1,2)}:=\Sigma_{W^{(1,2)}}$ be the covariance tensor associated to $W^{(1,2)}$ (see \eqref{eq:cov_tensor} for the definition). This is a non-negative linear operator acting on $\mathrm{Sym}_\beta(2N)$. Due to the $2\times 2$ block structure of $W^{(1,2)}$, the subspace $\mathrm{Sym}_\beta(N) \oplus \mathrm{Sym}_\beta(N)\subset \mathrm{Sym}_\beta(2N)$ is invariant under the action of $\Sigma^{(1,2)}$, which allows us to restrict $\Sigma^{(1,2)}$ to this subspace. Assumption \ref{ass:W-type}(iii) ensures that $\Sigma^{(1,2)}$ is invertible and, furthermore, that
\begin{equation}
\Sigma^{(1,2)}\gtrsim N^{-1}\alpha.
\label{eq:TS_bound}
\end{equation}
This condition is significantly weaker than the customary \emph{fullness} condition \cite[Assumption F]{slow_corr}, which is typically required in GFT arguments. In our setting the fullness condition corresponds to $\Sigma^{(1,2)}\gtrsim N^{-1}$, but this does not hold in the relevant regime when $\alpha$ is much smaller than 1. 

Our main contribution lies in the following matrix-valued flow interpolating between the initial pair of matrices and the pair of Gaussian divisible matrices. We define the GFT flow as the strong solution to the stochastic differential equation
\begin{equation}
\dif W^{(1,2)}_t = - \left(\beta N\Sigma^{(1,2)}\right)^{-1}\left[W^{(1,2)}_t\right]\dif t + N^{-1/2}\mathrm{diag}\left(\dif B^{(1)}_t, \dif B^{(2)}_t\right), \quad W^{(1,2)}_0=W^{(1,2)},
\label{eq:GFT_flow}
\end{equation}
where $\dif B^{(1)}_t$ and $\dif B^{(2)}_t$ are independent matrix-valued Brownian motions of the same symmetry class as $W^{(1,2)}$. Namely, the covariance structure of $B^{(j)}_t=\big(B^{(j)}_{t,ab}\big)_{a,b=1}^N$, $j=1,2$, is given by
\begin{equation*}
\E B^{(j)}_{t,ab}B^{(j)}_{t,cd}=tN^{-1}\delta_{ac}\delta_{bd}(1+\delta_{ab}),\quad \forall a,\,b,\,c,\,d\in[N]\,\,\text{with}\,\,a\le b,\, c\le d,
\end{equation*}
in the real symmetric case, and
\begin{equation*}
\E \ B^{(j)}_{t,ab}B^{(j)}_{t,cd} = tN^{-1}\delta_{ad}\delta_{bc},\quad \forall a,\,b,\,c,\,d\in[N],
\end{equation*}
in the complex Hermitian case. For instance, when $W^{(1)}, W^{(2)}$ are independent GOE/GUE matrices, \eqref{eq:GFT_flow} reduces to a pair of independent standard matrix-valued Ornstein-Uhlenbeck flows. While we have introduced the GFT flow in the context of two coupled matrices, resulting in the block-diagonal structure of $W^{(1,2)}$, this structure is not essential. This flow can be similarly defined for a single matrix.

Note that the drift term in \eqref{eq:GFT_flow} is singular, as the smallest eigenvalue of the operator $\Sigma^{(1,2)}$ may be very small. However, this term is significantly smaller than a naive estimate would suggest. To demonstrate this, denote for simplicity $W:=W^{(1,2)}$ and $\Sigma:=\Sigma^{(1,2)}$. To quantify the size of the drift term, consider, for instance, its Hilbert-Schmidt norm. A straightforward calculation relying on \eqref{eq:TS_bound} gives that
\begin{equation}
\E \langle \vert\Sigma^{-1}[W]\vert^2\rangle \lesssim \E \frac{N^2}{\alpha^2}\langle WW^*\rangle \lesssim \frac{N^2}{\alpha^2},
\label{eq:drift_naive}
\end{equation}
where in the final inequality we used \eqref{eq:init_flat} from Assumption \ref{ass:W-type}. To refine this estimate, we observe that due to \eqref{eq:cov_t_transform} the covariance tensor associated to $X:=\Sigma^{-1/2}[W]$ acts as the identity. In particular, the variances of the entries of $X$ are of order 1, and $\E\langle XX^*\rangle\lesssim N$. This yields
\begin{equation}
\E \langle \vert\Sigma^{-1}[W]\vert^2\rangle  = \E \langle \vert\Sigma^{-1/2}[X]\vert^2\rangle  \lesssim \E \frac{N}{\alpha}\langle XX^*\rangle \lesssim \frac{N^2}{\alpha},
\label{eq:drift}
\end{equation}
where we again used \eqref{eq:TS_bound}. Remarkably, the bound in \eqref{eq:drift} is $\alpha$ times better then the naive bound in \eqref{eq:drift_naive}. This interplay between $\Sigma^{-1}$ and $W$ is the key aspect of the flow \eqref{eq:GFT_flow} which allows us to run the comparison process for sufficiently long time while keeping the increments along the flow small.

The following statement summarizes the key properties of the flow \eqref{eq:GFT_flow}. The proof is given in Appendix~\ref{app:flow}.

\begin{lemma}[Properties of the GFT flow]\label{lemma:match} Under Assumption \ref{ass:W-type} the following holds.
\begin{enumerate}
\item Let $\Sigma_t^{(1,2)}$ be the covariance tensor associated to $W^{(1,2)}_t$. Then $\Sigma_t^{(1,2)}=\Sigma^{(1,2)}$ for any $t\ge 0$, i.e. the joint second order correlation structure is preserved under the flow \eqref{eq:GFT_flow}.
\item There exists a universal constant $c_*>0$ such that for any $t\in [0,\alpha]$ the solution to \eqref{eq:GFT_flow} can be represented in the form \eqref{eq:g_component} for some $s\in [c_*t, c_*^{-1}t]$.
\end{enumerate}
\end{lemma}

Note that the statement of Lemma \ref{lemma:match}(2) is valid only for $t\le\alpha$, meaning it holds only up to a relatively short time. This prevents performing a long-time comparison that would drive the initial pair of matrices to the purely Gaussian ones, as it was done in \cite{SchnelliXu23}. However this is not required for our purpose. We only need to run the GFT flow for a time exceeding $N^{-1}$, which can be chosen below $\alpha$ due to the bound $\alpha\ge N^{-1+\epsilon}$. 

Additionally to the properties listed in Lemma \ref{lemma:match}, the flow \eqref{eq:GFT_flow} preserves the block-diagonal structure of $W^{(1,2)}_t$, i.e.
\begin{equation*}
W^{(1,2)}_t =\mathrm{diag}\left(W^{(1)}_t,W^{(2)}_t\right)\quad\text{with}\quad W^{(j)}_t = \big(w^{(j)}_{t,ab}\big)_{a,b=1}^N,\,\, j=1,2,\,\, t\ge 0.
\end{equation*}
Since the only entries of $W^{(2)}$ that are correlated with $w^{(1)}_{ab}$ are $w^{(2)}_{ab}$ and $w^{(2)}_{ba}$, as stated in Assumption~\ref{ass:model}(iii), the GFT flow decomposes into $N(N+1)/2$ independent flows. Each of these flows couples $w^{(1)}_{t,ab}$ with $w^{(2)}_{t,ab}$ for some $1\le a\le b\le N$. For instance, in the real symmetric case these flows take the following form:
\begin{equation}
\dif \begin{pmatrix}
w^{(1)}_{t,ab}\\ w^{(2)}_{t,ab}
\end{pmatrix}
=-\frac{1+\delta_{ab}}{2N}S^{ab}\begin{pmatrix}
w^{(1)}_{t,ab}\\ w^{(2)}_{t,ab}
\end{pmatrix}\dif t
+\frac{1}{\sqrt{N}}\begin{pmatrix}
\dif B^{(1)}_{t,ab}\\ \dif B^{(2)}_{t,ab}
\end{pmatrix},
\label{eq:flow_entr}
\end{equation}
where $S^{ab}$ is the inverse of the covariance matrix of $w_{ab}^{(1)}$ and $w_{ab}^{(2)}$, i.e.
\begin{equation}
S^{ab}=
\begin{pmatrix}
s^{ab}_{11}&s^{ab}_{12}\\s^{ab}_{21}&s^{ab}_{22}
\end{pmatrix}
=\begin{pmatrix}
\E |w_{ab}^{(1)}|^2&\E w_{ab}^{(1)}w_{ab}^{(2)}\\
\E w_{ab}^{(1)}w_{ab}^{(2)}&\E |w_{ab}^{(2)}|^2
\end{pmatrix}^{-1}.
\label{eq:def_S_mat}
\end{equation}
In the complex Hermitian case the flow for $w_{ab}^{(1)}$ and $w_{ab}^{(2)}$ splits further into two independent flows, separately governing the real and imaginary parts of these entries.

The time-dependence of $W^{(1,2)}_t$ induces a corresponding time-dependence in the filtered signal
\begin{equation}
H^{(j)}_t:= A^{(j)}+\Phi^{(j)}\left[W^{(j)}_t\right],\quad j=1,2.
\label{eq:Ht}
\end{equation}
We further denote the joint eigenvalue correlation functions of $H^{(1)}_t$ and $H^{(2)}_t$ by $p^{(1,2)}_{t,m,n}$. The following continuity statement provides an upper bound on $p^{(1,2)}_{t,m,n}-p^{(1,2)}_{0,m,n}$ in the weak sense.

\begin{proposition}[GFT]\label{prop:p_comparison} Fix a (small) $\delta>0$. Under the assumptions and with notations from Theorem \ref{theo:main}, we have
\begin{equation}
\begin{split}
&\int_{\R^{m+n}} F({\bm x},{\bm y})\left(p^{(1,2)}_{t,m,n}\left(E^{(1)}+\frac{\bm{x}}{N\rho^{(1)}(E^{(1)})}, E^{(2)}+\frac{\bm{y}}{N\rho^{(2)}(E^{(2)})} \right)\right.\\
&\qquad \left.- p^{(1,2)}_{0,m,n}\left(E^{(1)}+\frac{\bm{x}}{N\rho^{(1)}(E^{(1)})}, E^{(2)}+\frac{\bm{y}}{N\rho^{(2)}(E^{(2)})}\right)\right)\dif\bm{x}\dif\bm{y} = \mathcal{O}\big(N^{-c(m,n)}\big)
\end{split}
\label{eq:p_comparison}
\end{equation}
uniformly in $0\le t\le \alpha^{1/2}N^{-1/2-\delta}$, where $c(m,n)>0$ depends only on $m,\,n$ and $\delta$, and the implicit constant in $\mathcal{O}$ depends on $F$.
\end{proposition}
\noindent The proof of Proposition \ref{prop:p_comparison} is presented in Section \ref{sec:GFT} and relies on the multivariate cumulant expansion \cite[Proposition 3.2]{slow_corr}. While Lemma \ref{lemma:match}(1) ensures the cancellation of the terms involving second order cumulants, the flow \eqref{eq:GFT_flow} also provides a powerful partial cancellation for the third order cumulants, see later in \eqref{eq:l=2_cancel}. This additional cancellation allows us to run the comparison process \eqref{eq:GFT_flow} up to the time $t=\alpha^{1/2}N^{-1/2-\delta}$.

In order to complete the proof of Theorem \ref{theo:main}, we bring Propositions \ref{prop:DBM} and \ref{prop:p_comparison} together. Take $t_0:=N^{-1+\epsilon/4}$. By Lemma \ref{lemma:match}, $W^{(j)}_{t_0}$ for $j=1,2$ admits the representation of the form \eqref{eq:g_component} with $s\sim t_0$. Consequently, Proposition \ref{prop:DBM} ensures that \eqref{eq:goal} holds for $p^{(1,2)}_{t_0,m,n}$. Moreover, since $t_0\sqrt{N/\alpha}\le N^{-\epsilon/4}$, using Proposition \ref{prop:p_comparison} for $p^{(1,2)}_{t_0,m,n}$, $p^{(1,2)}_{t_0,m,0}$ and $p^{(1,2)}_{t_0,0,n}$ we conclude the proof of Theorem \ref{theo:main}.

\section{Gaussian divisible case: Proof of Proposition \ref{prop:DBM}}\label{sec:GD}

First of all, we write $H^{(j)}$, $j=1,2$, in the form
\begin{equation}
H^{(j)}\!\!=\widehat{H}^{(j)}+\sqrt{t}W_G^{(j)}\quad \text{with}\quad\widehat{H}^{(j)}\!\!:= A^{(j)}+\Phi^{(j)} \big[ \widehat{W}^{(j)}\big] +\sqrt{s}\left(\big(\Phi^{(j)}\big)^2-(c_0/2)^2I\right)^{1/2}\!\!\big[ W^{(j)}_{1,G}\big],
\label{eq:H_divisible}
\end{equation}
$W_G^{(j)}:=W_{2,G}^{(j)}$ and $t:=c_0s/2$. Here the GUE/GOE matrices $\lbrace W^{(j)}_{i,G}\rbrace_{i,j=1}^2$ are independent, and the operator $I$ acts identically on $\mathrm{Sym}_\beta(N)$. The existence of this representation follows from the bound $\Phi^{(j)}\ge c_0$ given in Assumption~\ref{ass:model}.

The key tool for proving Proposition \ref{prop:DBM} is the fixed energy universality of the \emph{Dyson Brownian motion} established in \cite[Theorem~2.2]{fixed_E}. Informally, this result asserts that the local spectral statistics of a matrix $H+\sqrt{\tau}W_G$ is essentially determined by the Gaussian matrix $W_G$, provided that $\tau\gg 1/N$ and the eigenvalue distribution of $H$ is sufficiently regular. Notably, this holds even when $H$ is deterministic. To apply \cite[Theorem~2.2]{fixed_E}, we need the inputs summarized in the following lemma, which is proven later in Appendix \ref{app:free_conv}.

\begin{lemma}\label{lem:cond_rho} Fix $\kappa>0$ and let $j\in [2]$. Denote by $\widehat{\rho}^{(j)}$ the scDoS of $\widehat{H}^{(j)}$ defined in \eqref{eq:H_divisible}, and let $\widehat{\mathbf{B}}_\kappa^{(j)}$ be the corresponding $\kappa$-bulk. Under the assumptions of Proposition \ref{prop:DBM} the following holds.
\begin{enumerate}
\item We have $\mathbf{B}_\kappa^{(j)}\subset \widehat{\mathbf{B}}_{\kappa/2}^{(j)}$.
\item The local law stated in Proposition \ref{prop:loc_law} holds for the matrix $\widehat{H}^{(j)}$.
\item Let us condition on $\widehat{H}^{(j)}$ and interpret $H^{(j)}$ as a sum of a deterministic matrix and a Gaussian perturbation $\sqrt{t}W_G^{(j)}$. Let $\rho_c^{(j)}$ be the corresponding scDoS, which is a random variable defined on the probability space of $\widehat{H}^{(j)}$. Then for any $E\in\mathbf{B}_\kappa^{(j)}$ we have
\begin{equation}
\left\vert \rho_c^{(j)}(E)-\rho^{(j)}(E)\right\vert\prec \frac{1}{Nt}+t.
\label{eq:cond_rho}
\end{equation}
\end{enumerate}
\end{lemma}

\begin{proof}[Proof of Proposition \ref{prop:DBM}] For simplicity, we restrict the proof to the real symmetric case, as the argument remains identical in the complex Hermitian setting. Let $F\in C^\infty_c\left(\R^{m+n}\right)$ be a real-valued test function. Throughout the proof implicit constants may depend on $F$, however they appear only in sub-exponential factor. Additionally, we use $\varepsilon>0$ to denote an $N$-independent constant which can be chosen arbitrarily small, and $c>0$ to represent a universal constant, whose values may vary from line to line. 

Recalling the definition of $p^{(1,2)}_{m,n}$ from \eqref{eq:joint_cor_func} we get that
\begin{equation}
\begin{split}
&\int_{\R^{m+n}} F({\bf x},{\bf y}) p^{(1,2)}_{m,n}\left(E^{(1)}+\frac{{\bf x}}{\rho^{(1)}(E^{(1)})N}, E^{(2)}+\frac{{\bf y}}{\rho^{(2)}(E^{(2)})N}\right)\dif {\bf x}\dif {\bf y}\\
&\quad = C^{(1)}_{m}C^{(2)}_{n} \E\left[ \sum_{\,\,{\bm \lambda}^{(1)}\!,\, \bm{\lambda}^{(2)}} F\left(({\bm \lambda}^{(1)}-E^{(1)})N\rho^{(1)}(E^{(1)}), ({\bm \lambda}^{(2)}-E^{(2)})N\rho^{(2)}(E^{(2)}) \right)\right],
\end{split}
\label{eq:id_lhs}
\end{equation}
where we abbreviated 
\begin{equation*}
C^{(j)}_{k}:= \frac{(N-k)!N^k}{N!}\rho^{(j)}(E^{(j)})^{k}\quad \text{and}\quad {\bm \lambda}^{(1)}:=(\lambda_{i_1}^{(1)},\ldots, \lambda_{i_m}^{(1)})\in\R^{m},\, {\bm \lambda}^{(2)}:=(\lambda_{j_1}^{(2)},\ldots, \lambda_{j_n}^{(2)})\in\R^n.
\end{equation*}
The summation over $\lambda^{(1)}$ (respectively, over $\lambda^{(2)}$) in the rhs. of \eqref{eq:id_lhs} runs over all distinct indices $i_1,\ldots,i_m\in[N]$ (respectively, $j_1,\ldots,j_n\in[N]$), and the expectation is taken wrt. $(H^{(1)}, H^{(2)})$. We note that the normalizing factors $C^{(1)}_m$, $C^{(2)}_n$ are of order one.

We show that $\rho^{(j)}$ can be replaced by $\rho^{(j)}_c$ in the rhs. of \eqref{eq:id_lhs} by a cost of a small error term. Since $E^{(1)}$ lies in the $\kappa$-bulk of $\rho^{(j)}$, we have $\rho^{(j)}(E^{(j)})\ge \kappa$, and by Lemma \ref{lem:cond_rho}(3), $\rho^{(j)}_c(E^{(j)})\ge \kappa/2$ with very high probability. Moreover, since $F$ is compactly supported, only pairs $(\bm{\lambda}^{(1)},\bm{\lambda}^{(2)})$ satisfying $\lVert\bm{\lambda}^{(j)}-E^{(j)}\rVert\lesssim N^{-1}$ for $j=1,2$ contribute to this sum, as well as to the analogous sum where $\rho^{(j)}$ is replaced by $\rho_c^{(j)}$. By Proposition \ref{prop:loc_law}, $H^{(j)}$ satisfies the bulk local law, implying that the bulk eigenvalue rigidity from \cite[Corollary 2.5]{slow_corr} holds for $H^{(j)}$. Consequently, the number of terms contributing to the rhs. of \eqref{eq:id_lhs} is at most $N^\varepsilon$ with very high probability. Applying \eqref{eq:cond_rho} from Lemma \ref{lem:cond_rho} and recalling that $F$ is smooth, we conclude that the rhs. of \eqref{eq:id_lhs} equals to
\begin{equation}
C^{(1)}_m C^{(2)}_n \E\left[ \sum_{\,\,{\bm \lambda}^{(1)}\!,\, \bm{\lambda}^{(2)}}\!\! F\left(({\bm \lambda}^{(1)}-E^{(1)})N\rho_c^{(1)}(E^{(1)}), ({\bm \lambda}^{(2)}-E^{(2)})N\rho_c^{(2)}(E^{(2)}) \right)\right]\! +\! \mathcal{O}\left(N^{-c}\right).
\label{eq:rho_to_rho_fc}
\end{equation}

Next we condition on $\widehat{H}^{(1)}, \widehat{H}^{(2)}$ and $W^{(2)}_G$ in \eqref{eq:rho_to_rho_fc}, leaving the randomness only in $W^{(1)}_G$. By Lemma \ref{lem:cond_rho}(1), $E^{(1)}$ lies in the $\kappa/2$-bulk of $\widehat{H}^{(j)}$. Thus, applying Lemma \ref{lem:cond_rho}(2), we see that $\widehat{H}^{(j)}$ satisfies the local law \eqref{eq:loc_law} for spectral parameter $z_j$ with $\Re z_j\in (E^{(1)}-\varepsilon, E^{(1)}+\varepsilon)$. This implies that $\widehat{H}^{(1)}-E^{(1)}$ is $(N^{-1+\varepsilon},N^{-\varepsilon})$-regular in the sense of \cite[Definition 2.1]{fixed_E}. Therefore, for any fixed $\bm{\lambda}^{(2)}$ which contributes to the sum in \eqref{eq:rho_to_rho_fc}, the fixed energy universality of the Dyson Brownian motion from \cite[Theorem 2.2]{fixed_E} yields
\begin{equation}
\begin{split}
&\E\!\left[\E\! \left[C^{(1)}_m\sum_{\,\,\,{\bm \lambda}^{(1)}}\! F\left(({\bm \lambda}^{(1)}-E^{(1)})N\rho_c^{(1)}(E^{(1)}), ({\bm \lambda}^{(2)}-E^{(2)})N\rho_c^{(2)}(E^{(2)}) \right) \,\big|\, \widehat{H}^{(1)}\!,\, \widehat{H}^{(2)}\!,\, W^{(2)}_G  \right]\right]\\
&\quad = \E\int_{\R^m} F\left({\bf x}, ({\bm \lambda}^{(2)}-E^{(2)})N\rho_c^{(2)}(E^{(2)}) \right) p_m^{{\rm GOE}}\left(\frac{\bf{x}}{\rho_{\rm{sc}}(0)}\right) \dif {\bf x} + \mathcal{O}\left(N^{-c}\right),
\end{split}
\label{eq:1st_appl}
\end{equation}
where $p_m^{{\rm GOE}}$ denotes the $m$-point correlation function for the GOE. Now we sum \eqref{eq:1st_appl} over ${\bm \lambda}^{(2)}$. By the bulk rigidity of $\widehat{H}^{(2)}$ (as discussed above \eqref{eq:rho_to_rho_fc}), which implies that the number of nonzero terms in this sum is at most $N^\varepsilon$, we conclude that the error term $\mathcal{O}(N^{-c})$ in \eqref{eq:1st_appl} contributes at most $\mathcal{O}(N^{-c/2})$ to the sum. After conditioning on $\widehat{H}^{(2)}$ and applying \cite[Theorem~2.2]{fixed_E} for each fixed ${\bf{x}}\in\R^m$, we obtain that the sum in \eqref{eq:rho_to_rho_fc} equals to
\begin{equation*}
\mathcal{I}[F]:= \int_{\R^{m+n}} F({\bf x},{\bf y})p_m^{{\rm GOE}}\left(\frac{\bf{x}}{\rho_{\rm{sc}}(0)}\right)p_n^{{\rm GOE}}\left(\frac{\bf{y}}{\rho_{\rm{sc}}(0)}\right)\dif {\bf x}\dif {\bf y}
\end{equation*}
up to an error term $\mathcal{O}(N^{-c})$. Combining this with \eqref{eq:id_lhs}, we arrive to
\begin{equation}
\int_{\R^{m+n}} F({\bf x},{\bf y}) p^{(1,2)}_{m,n}\left(E^{(1)}+\frac{{\bf x}}{\rho^{(1)}(E^{(1)})N}, E^{(2)}+\frac{{\bf y}}{\rho^{(2)}(E^{(2)})N}\right)\dif {\bf x}\dif {\bf y}=\mathcal{I}[F]+\mathcal{O}\left(N^{-c}\right).
\label{eq:id1}
\end{equation}

In order to complete the proof of Proposition \ref{prop:DBM} it remains to show that
\begin{equation}
\int_{\R^{m+n}}\!\!\! F({\bf x},{\bf y}) p_m^{(1)}\!\left(\!E^{(1)}+\frac{{\bf x}}{\rho^{(1)}(E^{(1)})N}\right) p_n^{(2)}\!\left(\!E^{(2)}+\frac{{\bf y}}{\rho^{(2)}(E^{(2)})N}\right) \dif {\bf x}\dif {\bf y} = \mathcal{I}[F]+\mathcal{O}\left(N^{-c}\right).
\label{eq:id2}
\end{equation}
Let $S_1\subset \R^m$ and $S_2\subset \R^n$ be the projections of ${\rm supp}(F)\subset \R^m\times\R^n$ on the first $m$ and last $n$ coordinates, respectively. Denote the characteristic function of $S_j$ by $\mathcal{X}_{S_j}$. Since ${\rm supp}(F)$ is compact, $S_j$ is also compact, implying that $\mathcal{X}_{S_j}$ has finite integral. Firstly we show that $p_m^{(1)}$ can be replaced by $p_m^{{\rm GOE}}$ in the lhs. of \eqref{eq:id2} at the cost of an error term $\mathcal{O}(N^{-c})$. Following a similar argument as in the proof of \eqref{eq:id1}, we get
\begin{equation}
\int_{\R^m}\!\!\!\!F({\bf x},{\bf y}) p_m^{(1)}\!\left(E^{(1)}\!+\frac{{\bf x}}{\rho^{(1)}(E^{(1)})N}\right)\! \dif{\bf x} =\!\! \int_{\R^m}\!\!\!\!F({\bf x},{\bf y}) p_m^{{\rm GOE}}\!\left(\frac{{\bf x}}{\rho_{{\rm sc}}(0)N}\right)\! \dif{\bf x} + \mathcal{X}_{S_1}(y)\mathcal{O}\left(N^{-c}\right).
\label{eq:k_1_to_GOE}
\end{equation}
Integrating \eqref{eq:k_1_to_GOE} over ${\bf y}$ and applying the same reasoning to $p_n^{(2)}$, we conclude the proof of \eqref{eq:id2}.
\end{proof}

\section{GFT: proof of Proposition \ref{prop:p_comparison}}\label{sec:GFT}

We begin by observing the continuity of Green's functions along the flow \eqref{eq:GFT_flow}. 

\begin{proposition}[Green function comparison]\label{prop:GFT}
Fix (small) $\kappa>0$ and $m, n\in\N\cup\{0\}$. For a (small) $\xi>0$ take $z^{(1)}_p=E^{(1)}_p+\ii N^{-1-\xi}$, $p\in[m]$, and $z^{(2)}_q=E^{(2)}_q+\ii N^{-1-\xi}$, $q\in[n]$, with $E^{(1)}_p$ and $E^{(2)}_q$ in the $\kappa$-bulk of $H^{(1)}$ and $H^{(2)}$ respectively. Recall the notation $H_t^{(j)}$ from \eqref{eq:Ht} and denote $G^{(j)}_t(z):=(H^{(j)}_t-z)^{-1}$ for $j=1,2$. We set
\begin{equation}
R_t:=\prod_{p=1}^{m} \langle \Im G^{(1)}_t(z^{(1)}_p)\rangle\prod_{q=1}^{n} \langle\Im G^{(2)}_t(z^{(2)}_q)\rangle.
\label{eq:Rt}
\end{equation}
Then uniformly in $t\in[0,1]$ it holds that
\begin{equation}
\left\vert\E R_t - \E R_0\right\vert \lesssim \sqrt{\frac{N}{\alpha}}N^{C\xi} t
\end{equation}
for some constant $C>0$ which depends only on $\kappa, m$ and $n$.
\end{proposition}
The proof of Proposition \ref{prop:p_comparison} follows as a standard consequence of Proposition \ref{prop:GFT}, see e.g. \cite[Theorem 15.3]{erdHos2017dynamical}. Thus, we further focus on proving Proposition \ref{prop:GFT}.

\begin{proof}[Proof of Proposition \ref{prop:GFT}]  We focus on the real symmetric case, as the proof in the complex Hermitian setting closely follows the same lines. For $1\le a\le b\le N$ denote the derivative in $w^{(j)}_{ab}$ by $\partial^{(j)}_{ab}$ and differentiate $\E R_t$ along the flow \eqref{eq:GFT_flow}. Using \eqref{eq:flow_entr} and \eqref{eq:def_S_mat} along with the It\^{o} calculus we get that
\begin{equation*}
\partial_t\E R_t = \frac{1}{2N}\E\sum_{a\le b} \sum_{j\in[2]} (1+\delta_{ab})\left(- \partial_{ab}^{(j)}R_t \left( s_{j1}^{ab}w_{t,ab}^{(1)} +s_{j2}^{ab}w_{t,ab}^{(2)}\right) + \big(\partial_{ab}^{(j)}\big)^2R_t\right),
\end{equation*}
Note that the variance of $s_{j1}^{ab}w_{t,ab}^{(1)} +s_{j2}^{ab}w_{t,ab}^{(2)}$ is much smaller than the variance of each of the two constituting terms separately. Arguing similarly to \eqref{eq:drift_naive} -- \eqref{eq:drift} we have that
\begin{equation}
\E \big| s_{j1}^{ab}w_{t,ab}^{(1)}\big|^2 \lesssim \frac{N}{\alpha^2},\,\, \E \big| s_{j2}^{ab}w_{t,ab}^{(2)}\big|^2 \lesssim \frac{N}{\alpha^2},\quad \text{while}\quad \E \big| s_{j1}^{ab}w_{t,ab}^{(1)} +s_{j2}^{ab}w_{t,ab}^{(2)}\big|^2 \lesssim \frac{N}{\alpha}.
\label{eq:addit_smallness}
\end{equation}
Next we perform the multivariate cumulant expansion \cite[Proposition 3.2]{slow_corr} for 
\begin{equation*}
\E \big[w_{t,ab}^{(1)}\partial_{ab}^{(j)}R_t\big]\quad \text{and}\quad \E \big[w_{t,ab}^{(2)}\partial_{ab}^{(j)}R_t\big].
\end{equation*}
Abbreviating the joint cumulant $\kappa_{p,q}\big(w_{t,ab}^{(1)}, w_{t,ab}^{(2)}\big)$ by $\kappa_{t,ab}^{(p,q)}$ for non-negative integers $p$ and $q$, we arrive to
\begin{equation}
\begin{split}
\partial_t \E R_t =& -\frac{1}{2N}\E\sum_{a\le b} \sum_{j\in[2]} (1+\delta_{ab}) \sum_{l=1}^L\sum_{p+q=l}\frac{1}{p!q!}\left(s_{j1}^{ab}\kappa_{t,ab}^{(p+1,q)}+s_{j2}^{ab}\kappa_{t,ab}^{(p,q+1)}\right) \big(\partial_{ab}^{(1)}\big)^p \big(\partial_{ab}^{(2)}\big)^q \partial_{ab}^{(j)} R_t\\
&  + \frac{1}{2N}\E\sum_{a\le b} \sum_{j\in[2]} (1+\delta_{ab})\big(\partial_{ab}^{(j)}\big)^2R_t + \mathcal{O}(N^{-100})
\end{split}
\label{eq:R_expanded}
\end{equation}
for some positive integer $L$ which depends only on the model parameters from Assumptions \ref{ass:W-type} and \ref{ass:model}. We refer to the terms corresponding to $l$ in \eqref{eq:R_expanded} as the $l+1$-st order terms, as they involve the $l+1$-st order cumulants.

The rest of the proof is structured as follows. First, we consider the second order terms ($l=1$) in \eqref{eq:R_expanded} and demonstrate an algebraic cancellation between them and the first term in the second line of \eqref{eq:R_expanded}. Next, we establish an upper bound on terms of the fourth order and higher ($l\ge 3$) by estimating each of them individually. Finally, we analyze the third order terms ($l=2$), employing a more refined argument that exploits the cancellations between the entries of $S^{ab}$ and the third order cumulants.  

\underline{Second order terms ($l=1$).} Fix a pair of indices $a,b\in [N]$ with $a\le b$. An elementary calculation based on \eqref{eq:flow_entr} shows that the second order cumulants are preserved under the flow \eqref{eq:GFT_flow}. Therefore, for $p+q=2$ we simply write $\kappa_{ab}^{(p,q)}:=\kappa_{t,ab}^{(p,q)}$. The coefficient by $\partial_{ab}^{(1)}\partial_{ab}^{(2)}R_t$ in the first line of \eqref{eq:R_expanded} equals to
\begin{equation}
s_{21}^{ab}\kappa_{ab}^{(2,0)} + s_{22}^{ab}\kappa_{ab}^{(1,1)} + s_{11}^{ab}\kappa_{ab}^{(1,1)}+s_{12}^{ab}\kappa_{ab}^{(0,2)}.
\label{eq:coeff_l=1}
\end{equation}
From \eqref{eq:def_S_mat} it follows that $s_{12}^{ab}=s_{21}^{ab}$ and
\begin{equation}
\kappa^{(2,0)}_{ab}=\left(\det S^{ab}\right)^{-1}s_{22}^{ab},\quad \kappa^{(0,2)}_{ab}=\left(\det S^{ab}\right)^{-1}s_{11}^{ab},\quad \kappa^{(1,1)}_{ab}=-\left(\det S^{ab}\right)^{-1}s_{12}^{ab},
\label{eq:kappa_S}
\end{equation}
implying that the quantity in \eqref{eq:coeff_l=1} equals to zero. A similar calculation using \eqref{eq:kappa_S} shows that the remaining second order terms in the first line of \eqref{eq:R_expanded} cancel out with the first term in the second line of \eqref{eq:R_expanded}.

\underline{High order terms ($l\ge 3$).} To simplify the presentation, from now on we assume that all $z_p^{(1)}$ and $z_q^{(2)}$ are identical, i.e. $z^{(1)}_p=z_1$ and $z^{(2)}_q=z_2$ for $p\in[m]$, $q\in[n]$. We introduce the short-hand notation $G_{j}:=G^{(j)}_t(z_j)$, which simplifies \eqref{eq:Rt} to $R_t= \langle \Im G_1\rangle^m\langle \Im G_2\rangle^n$. Recall that $\Im z^{(1)}=\Im z^{(2)}=N^{-1-\xi}$ for some fixed $\xi>0$. We claim that
\begin{equation}
\left\vert\big(\partial_{ab}^{(1)}\big)^p \big(\partial_{ab}^{(2)}\big)^q \partial_{ab}^{(j)} R_t\right\vert\lesssim N^{C\xi}
\label{eq:deriv_bound}
\end{equation} 
holds uniformly in $a,b\in[N]$ with very high probability, for some constant $C>0$ depending only on $p+q$. First we show how to get the desired upper bound on the high order terms in \eqref{eq:R_expanded} and then proceed by proving \eqref{eq:deriv_bound}.

Assumption \ref{ass:W-type}(iii) ensures that $S^{ab}$ defined in \eqref{eq:def_S_mat} admits the upper bound 
\begin{equation}
S^{ab}\lesssim N\alpha^{-1}.
\label{eq:S_bound}
\end{equation}
In particular, $\vert s^{ab}_{ij}\vert\lesssim N\alpha^{-1}$ for all $a,b\in [N]$ and $i,j\in[2]$.  Using this bound together with \eqref{eq:deriv_bound} and \eqref{eq:p_moment}, we get that
\begin{equation}
\begin{split}
&\frac{1}{2N}\left\vert\E\sum_{a\le b} \sum_{j\in[2]} (1+\delta_{ab})\sum_{p+q=l}\frac{1}{p!q!}\left(s_{j1}^{ab}\kappa_{t,ab}^{(p+1,q)}+s_{j2}^{ab}\kappa_{t,ab}^{(p,q+1)}\right) \big(\partial_{ab}^{(1)}\big)^p \big(\partial_{ab}^{(2)}\big)^q \partial_{ab}^{(j)} R_t\right\vert\\
&\quad\lesssim \frac{1}{N} \sum_{a\le b}N\alpha^{-1}N^{-(l+1)/2} N^{C\xi}\lesssim \alpha^{-1}N^{-(l-3)/2}N^{C\xi}\lesssim \sqrt{N/\alpha}\,N^{C\xi},
\end{split}
\label{eq:l_high}
\end{equation}
where in the last inequality we used that $l\ge 3$ and $\alpha>N^{-1}$.

In order to prove \eqref{eq:deriv_bound} it is sufficient to show that for any non-negative integer $k$ it holds that
\begin{equation}
\left\vert\big(\partial_{ab}^{(j)}\big)^k \langle G_j\rangle\right\vert \lesssim N^{C\xi}
\label{eq:1G_deriv}
\end{equation}
uniformly in $a,b\in [N]$. The key ingredient in the proof of \eqref{eq:1G_deriv} is the following bound on $G_j$ for $j=1,2$ in the regime below the spectral resolution. Using the notation $\eta_j:=\vert\Im z_j\vert$, we assert that
\begin{equation}
\vert \langle G_j(z_j))\rangle\vert \prec \frac{1}{N\eta_j}, \quad \vert (G(z_j))_{ab}\vert \prec \frac{1}{N\eta_j}, \quad \vert (G^2(z_j))_{ab}\vert \prec \frac{1}{N\eta^2_j}
\label{eq:G_bounds}
\end{equation}
uniformly in $\Re z_j\in \mathbf{B}_\kappa^{(j)}$, $\eta_j\in (0,N^{-1})$ and $a,b\in[N]$. Although $H^{(j)}_t$ depends on the time parameter $t$, the corresponding $\rho^{(j)}$ remains fixed along the flow \eqref{eq:GFT_flow}, as it depends only on the second moments of $H^{(j)}_t$, which are preserved under this flow. In particular, $\mathbf{B}_\kappa^{(j)}$ does not depend on time. Thus, if $\Re z_j$ lies in the $\kappa$-bulk of $H^{(j)}_0$, it remains in the $\kappa$-bulk of $H^{(j)}_t$ for all $t\ge 0$. The bounds in \eqref{eq:G_bounds} are a standard consequence of the local law from Proposition \ref{prop:loc_law} and the monotonicity properties of the resolvent (see e.g. \cite[Lemma 5.3]{cusp_univ}).

For $k=0$, \eqref{eq:1G_deriv} follows directly from the first bound in \eqref{eq:G_bounds}. Now, assuming that $k>0$, we rewrite the lhs. of \eqref{eq:1G_deriv} as
\begin{equation*}
\left\vert\big(\partial_{ab}^{(j)}\big)^k \langle G_j\rangle\right\vert \sim \left\vert \left\langle G_j(\Phi^{(j)}[E_{ab}]G_j)^{k-1}\right\rangle\right\vert \le \frac{1}{N}\sum_{\bm{a},\bm{b}}\bigg|(G_j^2)_{b_{k-1}a_1}\prod_{i=1}^{k-1} (\Phi^{(j)}[E_{ab}])_{a_ib_i}\prod_{i=1}^{k-2}(G_j)_{b_ia_{i+1}}\bigg|,
\end{equation*}
where the summation runs over the sets of indices $\bm{a}=\{a_i\}_{i=1}^{k-1}\in [N]^k$ and $\bm{b}=\{b_i\}_{i=1}^{k-1}\in [N]^k$. Applying \eqref{eq:G_bounds} to bound the entries of $G_j$ and $G_j^2$, we get that
\begin{equation}
\left\vert\big(\partial_{ab}^{(j)}\big)^k \langle G_j\rangle\right\vert \lesssim  N^{C\xi}\bigg(\sum_{a',b'\in[N]} \big|(\Phi^{(j)}[E_{ab}])_{a'b'}\big|\bigg)^{k-1},
\end{equation}
which, together with the decay condition on the off-diagonal entries of  $\Phi^{(j)}[E_{ab}]$ stated in Assumption~\ref{ass:model}, immediately implies \eqref{eq:1G_deriv}.

\underline{Third order terms ($l=2$).} Since $w^{(1)}_{t,ab}$ and $w^{(2)}_{t,ab}$ have mean zero, their joint third order cumulants coincide with the corresponding joint moments. Thus, for any non-negative integers $p$ and $q$ with $p+q=2$ we have
\begin{equation*}
s_{j1}^{ab}\kappa_{t,ab}^{(p+1,q)}+s_{j2}^{ab}\kappa_{t,ab}^{(p,q+1)}\!\! = \E \left[\big(w_{ab}^{(1)}\big)^p \!\big(w_{ab}^{(1)}\big)^q\!\! \left(s_{j1}^{ab}w_{t,ab}^{(1)}+s_{j2}^{ab}w_{t,ab}^{(2)}\right)\right]\! = \!\E \left[\big(w_{t,ab}^{(1)}\big)^p \!\big(w_{t,ab}^{(1)}\big)^q \bm{e}_j^*S^{ab}\bm{w}_{t,ab}\right],
\end{equation*}
where $\{\bm{e}_1,\bm{e}_2\}$ is the canonical basis in $\C^2$ and $\bm{w}_{t,ab}=w^{(1)}_{t,ab}\bm{e}_1+w^{(2)}_{t,ab}\bm{e}_2$. Therefore,
\begin{equation}
\begin{split}
\left\vert s_{j1}^{ab}\kappa_{t,ab}^{(p+1,q)}+s_{j2}^{ab}\kappa_{t,ab}^{(p,q+1)}\right\vert \lesssim &\left(\E \,\big(w_{t,ab}^{(1)}\big)^{2p} \big(w_{t,ab}^{(1)}\big)^{2q}\right)^{1/2}\left(\E \bm{e}_j^*S^{ab}\bm{w}_{t,ab} \bm{w}_{t,ab}^* S^{ab}\bm{e}_j\right)^{1/2}\\
\lesssim &N^{-(p+q)/2} \left(\bm{e}_j^*S^{ab}\bm{e}_j\right)^{1/2}\lesssim N^{-1}\sqrt{N/\alpha}.
\end{split}
\label{eq:l=2_cancel}
\end{equation}
Here to go from the first to the second line we used that the covariance matrix of $\bm{w}_{t,ab}$ is given by $\big(S^{ab}\big)^{-1}$, and in the last inequality we also applied \eqref{eq:S_bound}. Individually $ s_{j1}^{ab}\kappa_{t,ab}^{(p+1,q)}$ and $s_{j2}^{ab}\kappa_{t,ab}^{(p,q+1)}$ have the size of order $N^{-1/2}\alpha^{-1}$, which is larger than their sum, as \eqref{eq:l=2_cancel} shows. This cancellation reflects the smallness of $s_{j1}^{ab}w_{t,ab}^{(1)} +s_{j2}^{ab}w_{t,ab}^{(2)}$ discussed in \eqref{eq:addit_smallness}. Combining \eqref{eq:l=2_cancel} with \eqref{eq:deriv_bound} we conclude that the lhs. of \eqref{eq:l_high} for $l=2$ has an upper bound of order $\sqrt{N/\alpha}\, N^{C\xi}$.

Having analyzed all terms in the rhs. of \eqref{eq:R_expanded}, we obtain
\begin{equation*}
\vert\partial_t \E R_t\vert\lesssim \sqrt{N/\alpha}\, N^{C\xi},
\end{equation*}
thereby completing the proof of Proposition \ref{prop:GFT}.
\end{proof}

\appendix

\section{Proofs of additional technical results}

\subsection{Optimality of Assumption \ref{ass:W-type}(iii): justification of Example \ref{ex:optimal}}\label{app:optimal} We embed $W^{(1)}$ and $W^{(2)}$ into the following diffusion processes with independent diffusion terms:
\begin{equation}
\dif W^{(j)}_t = N^{-1/2}\dif B^{(j)}_t,\quad W^{(1)}_0=W^{(2)}_0=W,
\label{eq:ex_flow}
\end{equation}
where $\dif B^{(1)}_t$ and $\dif B^{(2)}_t$ are defined as below \eqref{eq:GFT_flow}. At time $t=\alpha$, the matrices $W^{(1)}_t$ and $W^{(2)}_t$ have the same joint distribution as $W^{(1)}$ and $W^{(2)}$ defined in \eqref{eq:ex_W}. Analogously to \eqref{eq:Ht}, we define the time-dependent version of $H^{(j)}$ as $H_t^{(j)}:=A+W^{(j)}_t$, $j=1,2$. 

We now proceed as in Section \ref{sec:GFT}. To verify \eqref{eq:optimal}, it suffices to check that
\begin{equation}
\vert \E R_\alpha - \E R_0\vert \le N^{-c}
\label{eq:ex_R_difference}
\end{equation}
for some constant $c>0$, where $R_t$ is defined as in \eqref{eq:Rt}. Although proving \eqref{eq:optimal} only requires considering $m,n\in\{0,1,2\}$, our argument holds for any non-negative integers $m$ and $n$. Focusing on the real symmetric case for definiteness, we compute the derivative of $\E R_t$ along the flow \eqref{eq:ex_flow}:
\begin{equation}
\partial_t \E R_t = \frac{1}{2N}\E\sum_{a\le b}\sum_{j\in[2]} (1+\delta_{ab}) \big(\partial_{ab}^{(j)}\big)^2 R_t,
\label{eq:ex_deriv}
\end{equation}
where $\partial_{ab}^{(j)}$ denotes the derivative with respect to $w^{(j)}_{ab}$. Using \eqref{eq:ex_deriv} together with \eqref{eq:deriv_bound} in the special case where $\Phi^{(1)}=\Phi^{(2)}$ act identically on $\mathrm{Sym}_\beta(N)$, we obtain
\begin{equation*}
\vert\partial_t \E R_t\vert\lesssim N^{1+C\xi}.
\end{equation*}
Since $\alpha\le N^{-1-\epsilon}$ and $\xi>0$ can be chosen arbitrarily small, this completes the proof of \eqref{eq:ex_R_difference} and thus the proof of \eqref{eq:optimal}.\qed

\subsection{Local law: proof of Proposition \ref{prop:loc_law} }\label{app:loc_law} The local law stated in Proposition \ref{prop:loc_law} follows immediately from \cite[Theorem 2.8]{cusp_univ}, which establishes the local law for a broad class of correlated random matrix. To complete the proof, it remains to verify that the assumptions of \cite[Theorem 2.8]{cusp_univ} hold for our model \eqref{eq:model}, which we now proceed to do.

Assumption 2.1 from \cite{cusp_univ} coincides with Assumption \ref{ass:W-type}(i), while \cite[Assumption 2.2]{cusp_univ} is an analogue of \eqref{eq:p_moment}, but stated for the entries of $\widetilde{W}^{(j)}=\big(\widetilde{w}^{(j)}_{ab}\big)_{a,b=1}^N$, $j=1,2$. To verify this bound, for any $p\in\N$ and $a,b\in[N]$ we compute that
\begin{equation*}
\begin{split}
\E \big| \widetilde{w}^{(j)}_{ab}\big|^p & = \E\bigg| \sum_{i\le k}\left( \big(\Phi^{(j)}[E_{ik}]\big)_{ab} \Re w^{(j)}_{ik} + \big(\Phi^{(j)}[E_{ik}^\Im]\big)_{ab} \Im w^{(j)}_{ik} \right)\bigg|^p\\
&\lesssim N^{-p/2} \bigg(\sum_{i\le k}\big| \big(\Phi^{(j)}[E_{ik}]\big)_{ab}\big| + \big| \big(\Phi^{(j)}[E_{ik}^\Im]\big)_{ab}\big|\bigg)^p\lesssim N^{-p/2},
\end{split}
\end{equation*}
where the implicit constants depend only on $p$. The transition from the first to the second line follows from \eqref{eq:p_moment}, while in the second line we additionally use \eqref{eq:Phi_decay} and its complex Hermitian counterpart. Furthermore, the fullness condition from \cite[Assumption 2.4]{cusp_univ} follows directly from the inequality $\Phi^{(j)}\ge c_0$, and the upper bound on $\|M^{(j)}\|$ from \cite[Assumption 2.5]{cusp_univ} holds as we are working in the bulk regime. 

Now we turn to Assumption 2.3 from \cite{cusp_univ}. For clarity, we present the argument in the real symmetric case, the proof in complex Hermitian case follows analogously. We verify the first part of this assumption in the simplified form formulated in \cite[Example 2.6]{cusp_univ}, namely by showing that the joint cumulants of the entries of $\widetilde{W}^{(j)}$ exhibit a tree-like decay. Let $k\ge 2$ be a positive integer, and consider the index pairs $\alpha_i=(a_i,b_i)\in [N]^2$ for $i\in [k]$, where $a_i\le b_i$. Since the entries of $W^{(j)}$ are independent up to the symmetry constraint, we obtain
\begin{equation}
\kappa(\sqrt{N}w_{\alpha_1},\ldots,\sqrt{N}w_{\alpha_k}) =\sum_\beta \kappa_k\big(\sqrt{N}w^{(j)}_\beta\big)\prod_{i=1}^k \big( \Phi^{(j)}[E_\beta]\big)_{\alpha_i},
\end{equation}
where the summation runs over all $\beta=(a,b)\in [N]^2$ with $a\le b$. Applying \eqref{eq:p_moment} to bound the cumulants of $w^{(j)}_\beta$ and combining this with \eqref{eq:Phi_decay}, we derive
\begin{equation*}
\vert \kappa(\sqrt{N}w_{\alpha_1},\ldots,\sqrt{N}w_{\alpha_k})\vert \lesssim \prod_{e\in\mathfrak{T}_{\mathrm{min}}}\frac{1}{1+d(e)^s},
\end{equation*} 
where $\mathfrak{T}_{\mathrm{min}}$ is a minimal spanning tree as defined in \cite[Example 2.6]{cusp_univ} Here for an edge $e=(\alpha_i,\alpha_l)$ we denoted $d(e):=|a_i-a_l|+|b_i-b_l|$. This completes the verification of \cite[Assumption 2.3(i)]{cusp_univ}.

Finally, we demonstrate that the second part of \cite[Assumption 2.3]{cusp_univ} is not needed in our setting. This assumption states that each entry of $\widetilde{W}^{(j)}$ for $j=1,2$ is correlated with at most $N^{1/2-\mu}$ other entries for some fixed $\mu>0$. However, this condition does not necessarily hold in our case, as we allow each entry to be correlated with all others, provided that the correlations decay sufficiently fast, as specified in \eqref{eq:Phi_decay}. In \cite{cusp_univ}, Assumption 2.3(ii) was introduced solely to justify the truncation of the multivariate cumulant expansion in \cite[Proposition 5.2]{cusp_univ}. However, instead of expanding in terms of the entries of $\widetilde{W}^{(j)}$, we carry out the cumulant expansion directly with respect to the entries of $W^{(j)}$, $j=1,2$. This approach allows us to truncate the resulting series easily, since only a few entries of $W^{(1)},W^{(2)}$ exhibit correlations. Crucially, when rewriting the series in terms of $\widetilde{W}^{(j)}$, we recover the cumulant expansion with respect to $\widetilde{W}^{(j)}$. Thus, we conclude the proof of Proposition \ref{prop:loc_law}.\qed

\subsection{Properties of the GFT flow: proof of Lemma \ref{lemma:match}}\label{app:flow} The first statement of Lemma \ref{lemma:match} follows by a straightforward It\^{o} calculus. Next, we consider the real symmetric case and base our proof of Lemma \ref{lemma:match}(2) on \eqref{eq:flow_entr}. The complex Hermitian case follows similarly, requiring only a minor adjustment: using the the analogues of \eqref{eq:flow_entr} for the real and imaginary parts of $w^{(1)}_{ab}$ and $w^{(2)}_{ab}$. This is the only necessary modification, so we omit further details.

We construct the representation \eqref{eq:g_component} independently for each pair of indices $a,b\in [N]$, where $a\le b$. For simplicity, we further assume that $a<b$ to avoid introducing the scaling factor $1+\delta_{ab}$ in \eqref{eq:flow_entr}, the case $a=b$ follows identically. Let $O^{ab}$ be a $2\times 2$ orthogonal matrix that diagonalizes $S^{ab}$, i.e. $S^{ab} = O^{ab}\Lambda^{ab}(O^{ab})^*$, where $\Lambda^{ab}$ is a $2\times 2$ diagonal matrix with diagonal entries $\Lambda^{ab}_{11}$ and $\Lambda^{ab}_{22}$. Following the discussion above \eqref{eq:l=2_cancel}, we define $\bm{w}_{t,ab}:=\big(w_{t,ab}^{(1)},w_{t,ab}^{(2)}\big)$. Although $\bm{w}_{t,ab}$ is written as a row for typographical convenience, we treat it as a column. This convention is maintained throughout this section. Additionally, we denote 
\begin{equation*}
\bm{\xi}_{t,ab}:=(O^{ab})^* \bm{w}_{t,ab} = \big(\xi_{t,ab}^{(1)},\xi_{t,ab}^{(2)}\big)\quad \text{and}\quad \bm{B}_{t,ab}:=\big(B^{(1)}_{t,ab},B^{(2)}_{t,ab}\big).
\end{equation*}
Applying the orthogonal transformation $(O^{ab})^*$ from the left to \eqref{eq:flow_entr}, we obtain
\begin{equation}
\dif \bm{\xi}_{t,ab} = -(2N)^{-1}\Lambda^{ab} \bm{\xi}_{t,ab} \dif t +N^{-1/2}(O^{ab})^* \dif \bm{B}_{t,ab}.
\label{eq:xi_flow}
\end{equation}
Since the distribution of standard 2-dimensional Brownian motion remains unchanged under orthogonal transformations, \eqref{eq:xi_flow} decomposes into two flows with independent diffusion terms. Consequently, for any $t\ge 0$ we have that
\begin{equation}
\xi^{(j)}_{t,ab}\stackrel{d}{=} \mathrm{exp}\{ -(2N)^{-1}\Lambda^{ab}_{jj}t\}\xi^{(j)}_{0,ab}+ \sigma^{(j)}_{t,ab}\xi_{G,ab}^{(j)}\quad \text{with}\quad \sigma^{(j)}_{t,ab}=\sqrt{(1-\mathrm{exp}\{-N^{-1}\Lambda_{jj}^{ab}t\})/\Lambda^{ab}_{jj}},
\label{eq:xi_t}
\end{equation}
where the standard Gaussians $\xi^{(1)}_{G,ab}$ and $\xi^{(2)}_{G,ab}$ are independent and also independent of the initial condition $\bm{\xi}_{0,ab}$. Recalling \eqref{eq:S_bound}, we conclude that there exists a constant $c_*$ depending only on the model parameters in Assumptions \ref{ass:W-type} and \ref{ass:model}, such that
\begin{equation}
c_*t\le N\big(\sigma_{t,ab}^{(j)}\big)^2\le c_*^{-1}t
\end{equation}
for any $t\in [0,\alpha]$, $a,b\in [N]$ and $j=1,2$. By splitting the second term in \eqref{eq:xi_t} into two independent Gaussian components, we obtain the representation
\begin{equation}
\xi^{(j)}_{t,ab}=\widehat{\xi}^{(j)}_{ab}+\sqrt{s}N^{-1/2}\widehat{\xi}_{G,ab}^{(j)},\quad j=1,2,
\label{eq:xi_rep}
\end{equation}
for some $s\in [c_*t,c_*^{-1}t]$ which is the same for all indices $a$ and $b$. Here the independent standard Gaussians $\widehat{\xi}_{G,ab}^{(j)}$, $j=1,2$, are also independent of $\widehat{\xi}^{(1)}_{ab}$ and $\widehat{\xi}^{(2)}_{ab}$, which satisfy moment conditions \eqref{eq:init_flat} and \eqref{eq:p_moment}. Moreover, since $S^{ab}$ is an inverse of the covariance matrix of $w^{(j)}_{0,ab}$, $j=1,2$, we have $\E \xi^{(1)}_{0,ab}\xi^{(2)}_{0,ab}=0$ and this property is inherited by $\widehat{\xi}^{(j)}_{ab}$. Therefore, applying the orthogonal transformation $O^{ab}$ to $\bm{\xi}_{t,ab}$ expressed in the form \eqref{eq:xi_rep}, yields the desired representation \eqref{eq:g_component} of $W_t^{(j)}$.\qed

\subsection{Proof of Lemma \ref{lem:cond_rho}}\label{app:free_conv} The first statement of Lemma \ref{lem:cond_rho} follows directly from \cite[Proposition 10.1(b)]{shape}. The second statement follows by a slight modification of the proof of Proposition \ref{prop:loc_law} presented in Appendix \ref{app:loc_law}. In that proof we already verified that the correlation structure of $\Phi^{(j)}\big[ \widehat{W}^{(j)}\big]$ satisfies the assumptions of \cite[Theorem 2.8]{cusp_univ}. Additionally, we have a polynomial decay of the correlations in
\begin{equation*}
\left(\big(\Phi^{(j)}\big)^2-(c_0/2)^2I\right)^{1/2}\big[ W^{(j)}_{1,G}\big].
\end{equation*} 
For the second order cumulants, this is ensured by Assumption \ref{ass:model}, while higher order cumulants vanish since $W^{(j)}_{1,G}$ is Gaussian. Combined with the independence of $\widehat{W}^{(j)}$ and $ W^{(j)}_{1,G}$, this shows that $\widehat{H}^{(j)}$ satisfies the assumptions of \cite[Theorem 2.8]{cusp_univ}. Therefore, the local law \eqref{eq:loc_law} holds for $\widehat{H}^{(j)}$.

Now we proceed with the proof of \eqref{eq:cond_rho}. Since all quantities below depend on $j\in [2]$, we fix $j$ and omit it from the notation for brevity. First, we introduce the short-hand notation
\begin{equation*}
\widehat{m}_N(z):=\big\langle (\widehat{H}-z)^{-1}\big\rangle,\quad z\in\C\setminus\R.
\end{equation*}
Let $\widehat{M}$ and $M_c$ be the solutions to the MDE \eqref{eq:MDE} corresponding to $\widehat{H}$ and $\widehat{H}+\sqrt{t}W_G$ (conditioned on $\widehat{H}$, i.e. with $A=\widehat{H}$ in \eqref{eq:MDE}), respectively. We set
\begin{equation*}
\widehat{m}(z):=\big\langle \widehat{M}(z)\big\rangle \quad\text{and}\quad m_c(z):=\big\langle M_c(z)\big\rangle.
\end{equation*}
Since the bulk local law holds for $\widehat{H}$, \cite[Lemma A.1]{fixed_E} implies that there exist $N$-independent constants $c,C>0$ such that  $c<\Im m_c(z)<C$ for any $z\in\C\setminus\R$ with $\Re z$ in the $\kappa/3$-bulk of $\widehat{\rho}$ and $\vert\Im z\vert\in [0,1]$. It is known from \cite{Biane} that $m_c$ and $\widehat{m}_N$ are related as follows
\begin{equation}
m_c(z) = \widehat{m}_N (z+t m_c(z)),\quad z\in \C\setminus\R.
\label{eq:fc_def}
\end{equation}
Take $z_0:=E+\ii 0$ in \eqref{eq:fc_def}, where $E$ lies in the $\kappa$-bulk of $\rho$. Then $E$ is also in the $\kappa/2$-bulk of $\widehat{\rho}$ by part~(1), so 
\begin{equation*}
\Im (z_0+tm_c(z_0))=t\Im m_c(z_0)\sim t\ge N^{-1+\sigma}\quad \text{and}\quad \Re (z_0+tm_c(z_0)) = E+\mathcal{O}(t)\in \widehat{\mathbf{B}}_{\kappa/3}.
\end{equation*}
Therefore, using \eqref{eq:fc_def} and applying the local law from Lemma \ref{lem:cond_rho}(2) to $\widehat{m}_N(z_0+tm_c(z_0))$, we obtain
\begin{equation}
m_c(z_0) = \widehat{m}(z_0+t m_c(z_0)) + \mathcal{O}_\prec\left((Nt)^{-1}\right) = \widehat{m}(z_0) + \mathcal{O}_\prec\left(t m_c(z_0)+ (Nt)^{-1}\right),
\label{eq:m_c_comp}
\end{equation}
where in the last step we used that by \cite[Proposition 4.7, (4.20)]{shape} the derivative of $\widehat{m}(z)$ has an upper bound of order 1 for $\Re z$ in the bulk of $\widehat{\rho}$. Rewriting \eqref{eq:m_c_comp} as
\begin{equation*}
m_c(z_0) (1-t\mathcal{O}_\prec(1)) = \widehat{m}(z_0) + \mathcal{O}_\prec\left((Nt)^{-1}\right),
\end{equation*}
and using the bound $\vert \widehat{m}(z_0)\vert\lesssim 1$ from \cite[Proposition 3.5, (3.12)]{shape}, we complete the proof of Lemma~\ref{lem:cond_rho}.\qed

\bibliographystyle{plain} 
\bibliography{refs}

\end{document}